\documentclass[12pt]{article}
\usepackage{amsmath,amssymb,amsthm,amsfonts,graphicx,xcolor}
\usepackage{mathrsfs}
\usepackage{indentfirst}
\usepackage{chngcntr}
\usepackage{bbm}
\usepackage[blocks]{authblk}
\usepackage{newtxmath}
\usepackage{cite}
\usepackage[colorlinks=true,linkcolor=blue,citecolor=blue,urlcolor=blue]{hyperref}
\usepackage{xurl}
\allowdisplaybreaks[2]
\usepackage{geometry}
\usepackage{verbatim}
\DeclareMathOperator{\supp}{supp}
\newcommand{\doi}[1]{\href{https://doi.org/#1}{\nolinkurl{doi:#1}}}

\begin{document}
\renewcommand{\a}{\alpha}
\newcommand{\D}{\Delta}
\newcommand{\ddt}{\frac{\mathrm{d}}{\mathrm{d}t}}
\counterwithin{equation}{section}
\newcommand{\e}{\epsilon}
\newcommand{\eps}{\varepsilon}
\newtheorem{theorem}{Theorem}[section]
\newtheorem{proposition}{Proposition}[section]
\newtheorem{lemma}[theorem]{Lemma}
\newtheorem{remark}[theorem]{Remark}
\newtheorem{example}{Example}[section]
\newtheorem{definition}{Definition}[section]
\newtheorem{corollary}[theorem]{Corollary}
\makeatletter
\newcommand{\rmnum}[1]{\romannumeral #1}
\newcommand{\Rmnum}[1]{\expandafter\@slowromancap\romannumeral #1@}
\makeatother

\title{Observability Inequalities and the Logvinenko--Sereda Theorem for the Dunkl Transform}
\author{Xingyu Zhao$^{1}$\footnote{Corresponding author.}, Longben Wei$^{2}$, Zhiwen Duan$^{1}$ \\
			{\small {\it $^{1}$ School of Mathematics and Statistics, Huazhong University of Science }}\\
			{\small {\it and Technology, Wuhan, {\rm 430074,} P.R.China}}  \\
			{\small {\it $^{2}$ School of Mathematical Sciences, Guizhou Normal University, Guiyang 550025, P.R. China}}\\
			{\small {\it Email: zhaoxingyu@hust.edu.cn~(X.Zhao)}}\\
			{\small {\it Email: longbenwei51@gmail.com~(L.Wei)}}\\
			{\small {\it Email: duanzhw@hust.edu.cn~(Z.Duan)}}}

\date{September 17, 2026}
\maketitle

\begin{abstract} 
Let $R$ be a normalized root system in $\mathbb{R}^d$ with reflection group $G$, let $k$ be a $G$-invariant multiplicity function, and let $\mathcal{F}_k$ be the associated Dunkl transform. We write $k_\alpha=k(\alpha)>0$ and $\mathrm{d}\mu_k(x)=w(x)\,\mathrm{d}x$, where $w(x)=\prod_{\alpha\in R_+}|\langle x,\alpha\rangle|^{2k_\alpha}$.
 
We study observability, H\"older-type interpolation, and spectral inequalities for the Dunkl heat equation on $\mathbb{R}^d$. We establish a Bernstein inequality for ordinary derivatives and a Logvinenko--Sereda theorem with a spectral constant of the form $e^{C(1+N)}$ for functions whose Dunkl transforms are supported in $\overline{B(0,N)}$. We characterize observable sets as the measurable sets that are thick with respect to $\mu_k$, and prove the equivalence of the observability, H\"older-type interpolation, and spectral inequalities.
\end{abstract}


\section{Introduction}
% [R8-LANG-003] Removed redundant introductory wording.
We consider the Cauchy problem for the Dunkl heat equation:
\begin{equation}\label{EHEAT}
	\begin{cases}
		\partial_tu(t,x)-\Delta_ku(t,x)=0,
		& (t,x)\in(0,\infty)\times\mathbb{R}^d,\\
		u(0,x)=u_0(x),
		& u_0\in L_k^2(\mathbb{R}^d),
	\end{cases}
\end{equation}
Our aim is to establish quantitative unique continuation inequalities for the Dunkl heat equation \eqref{EHEAT}, where $L_k^2(\mathbb{R}^d)$ is the weighted space associated with the measure
% [R8-NOTATION-G001] Unified multiplicity/derivative/space/index notation or corrected a typographical inconsistency.
$ \mathrm{d}\mu_k(x)=w(x)\,\mathrm{d}x $, and $w(x)=\prod_{\alpha\in R_+}|\langle x,\alpha\rangle|^{2k_\alpha}$. Here
\[
\Delta_k=\sum_{j=1}^d T_j^2
\]
is the Dunkl Laplacian, with $\{ T_j\}_{j=1}^d$ denoting the Dunkl operators associated with a finite reflection group $G$ and a strictly positive multiplicity function $k$. Dunkl operators, introduced by Dunkl in \cite{CFD1}, are differential--difference operators naturally associated with finite reflection groups.

% [R8-LANG-004] Specified the core and distinguished the operator from its closure; nonnegative replaces positive.
The operator $-\Delta_k$, initially defined on $C_c^\infty(\mathbb{R}^d)$, is symmetric, nonnegative, and essentially self-adjoint on $L_k^2(\mathbb{R}^d)$; see \cite{AB}. We use the same notation for its self-adjoint closure. The resulting heat semigroup $H_t=e^{t\Delta_k}$ is self-adjoint and contractive on $L_k^2(\mathbb{R}^d)$. Its positivity follows from the maximum principle established in \cite{MR2}; precise definitions are recalled in Section~2.

We begin by introducing some notation.

% [R8-NOTATION-005] Standardized Schwartz, test-function, derivative, local norm, and essential-support notation.
\textit{Notation.} The symbols $C(\cdots)$, $C'(\cdots)$, and $C_j(\cdots)$ denote positive constants depending only on the indicated parameters; their values may change from one occurrence to the next. We write $B(x,r)$ for the open Euclidean ball with center $x$ and radius $r>0$, and $\mathbb{S}^{d-1}$ for the unit sphere. We set $\mathbb{N}=\{0,1,2,\dots\}$ and $\mathbb{N}^+=\mathbb{N}\setminus\{0\}$. Unless specified otherwise, $Q=(-1/2,1/2)^d$, and $x+LQ=\{x+Ly:y\in Q\}$ is the cube of side length $L>0$ centered at $x$. For $f\in L^2(\mathbb{R}^d)$, $\widehat f$ denotes its Fourier transform. The notation $\supp f$ refers to the essential closed support of a measurable function. We use $|x|=(\sum_{j=1}^d x_j^2)^{1/2}$ and $\langle x\rangle=(1+|x|^2)^{1/2}$, and write $\partial_j=\partial/\partial x_j$ and $\partial^\beta=\partial_1^{\beta_1}\cdots\partial_d^{\beta_d}$ for $\beta\in\mathbb{N}^d$. The Schwartz space is denoted by $\mathcal{S}(\mathbb{R}^d)$, and $C_c^\infty(\Omega)$ denotes the space of smooth compactly supported functions on an open set $\Omega\subset\mathbb{R}^d$.

For a measurable set $D\subset\mathbb{R}^d$, we write $|D|$ for its Lebesgue measure, $D^c$ for its complement, and $\mu_k(D)=\int_D\mathrm{d}\mu_k(x)$. The weighted spaces $L_k^p(D)=L^p(D,\mathrm{d}\mu_k)$ have norms
\begin{align*}
\|f\|_{L_k^p(D)}&=\left(\int_D|f(x)|^p\,\mathrm{d}\mu_k(x)\right)^{1/p},\qquad 1\le p<\infty,\\
\|f\|_{L_k^\infty(D)}&=\operatorname{ess\,sup}_{x\in D}|f(x)|.
\end{align*}
When $D=\mathbb{R}^d$, we may omit the domain. Since $w>0$ almost everywhere, the weighted and Lebesgue essential suprema agree; we therefore also use $L^\infty(D)$. For any measurable set $A$, $\chi_A$ denotes its characteristic function.

We first define relatively dense sets with respect to the weighted measure.
\begin{definition}\label{D1}
A measurable subset $E\subset\mathbb{R}^d$ is said to be relatively dense (or thick) with respect to $\mu_k$ if there exist $\gamma\in(0,1]$ and $L>0$ such that, for every $x\in\mathbb{R}^d$,
\begin{equation}
\mu_k(E\cap(x+LQ))\ge\gamma\,\mu_k(x+LQ),
\qquad Q=(-1/2,1/2)^d.
\end{equation}
Here $L$ is the side length of the cube.

In this case, $E$ is called a $(\gamma,L)$-relatively dense subset with respect to $\mu_k$.
\end{definition}
% [R8-LANG-006] Removed repetition and stated the parameter dependence proved in Section 3.
Section~3 shows that thickness with respect to Lebesgue measure is equivalent to thickness with respect to $\mu_k$, with the side length unchanged and the density parameter adjusted.
\begin{theorem}\label{TEV}
% [R8-MATH-007] Corrected the quantifiers in Theorem 1.1; its proof preserves the side length.
Let $E\subset\mathbb{R}^d$ be measurable. Then $E$ is thick with respect to $\mu_k$ if and only if it is thick with respect to Lebesgue measure. More precisely, in either direction, $(\gamma,L)$-thickness implies $(\widetilde\gamma,L)$-thickness for the other measure, where $\widetilde\gamma\in(0,1]$ depends only on $\gamma,d,R_+$, and $k$.
\end{theorem}

Next, we introduce an observability inequality, an interpolation inequality, and a spectral inequality for the Dunkl heat equation \eqref{EHEAT}.

\emph{The observability inequality.} A measurable set $E \subset \mathbb{R}^d$ is said to satisfy the observability inequality for equation \eqref{EHEAT} if, for every $T > 0$, there exists a positive constant $C_{\mathrm{obs}} = C_{\mathrm{obs}}(d,R,k,T,E)$ such that whenever $u$ solves \eqref{EHEAT},
\begin{equation}\label{EOBS}
\int_{\mathbb{R}^d} |u(T, x)|^2 \mathrm{d}\mu_k(x) \le C_{\mathrm{obs}} \int_{0}^{T} \int_{E} |u(t, x)|^2 \mathrm{d}\mu_k(x) \mathrm{d}t.
\end{equation}

\emph{The H\"older-type interpolation inequality.} A measurable set $E \subset \mathbb{R}^d$ is said to satisfy the H\"older-type interpolation inequality for the heat equation \eqref{EHEAT} if, for every $\theta \in (0,1)$, there is a constant $C_{\mathrm{Hold}} = C_{\mathrm{Hold}}(d,R,k,E,\theta)$ such that for every $T > 0$ and every solution $u$ of equation \eqref{EHEAT},
\begin{equation}\label{EHOLDER}
\int_{\mathbb{R}^d} |u(T, x)|^2 \mathrm{d}\mu_k(x)
\le e^{C_{\mathrm{Hold}}\left(1+\frac{1}{T}\right)}
\left( \int_{E} |u(T, x)|^2 \mathrm{d}\mu_k(x) \right)^{\theta}
\left( \int_{\mathbb{R}^d} |u(0, x)|^2 \mathrm{d}\mu_k(x) \right)^{1-\theta}.
\end{equation}

\emph{The spectral inequality.} A measurable set $E \subset \mathbb{R}^d$ is said to satisfy the spectral inequality if there is a positive constant $C_{\mathrm{spec}} = C_{\mathrm{spec}}(d,R,k,E)$ such that for every $N > 0$,
\begin{equation}\label{ESPEC}
\int_{\mathbb{R}^d} |f(x)|^2 \mathrm{d}\mu_k(x)
\le e^{C_{\mathrm{spec}}(1+N)}
\int_{E} |f(x)|^2 \mathrm{d}\mu_k(x)
\end{equation}
for every $f\in L_k^2(\mathbb{R}^d)$ with $\supp\mathcal{F}_k f\subset \overline{B(0,N)}$.

When a measurable set $E\subset \mathbb{R}^d$ satisfies \eqref{EOBS}, it is called an observable set for \eqref{EHEAT}.
% [R8-LANG-008] Replaced an informal recovery interpretation by the precise norm-estimate interpretation.
The inequality bounds the terminal-state norm by the observations on $E$ over $(0,T)$ and thus expresses stable determination of the terminal state from these observations.
% [R8-LANG-009] Combined repetitive sentences.
The interpolation inequality \eqref{EHOLDER} gives quantitative unique continuation in the form of H\"older-type propagation of smallness. In fact, if
\[
\int_{E} |u(T,x)|^{2}\mathrm{d}\mu_k(x)=\delta,
\]
then \eqref{EHOLDER} shows that
\[
\int_{\mathbb{R}^{d}} |u(T,x)|^{2}\mathrm{d}\mu_k(x)
\]
is bounded by $C\delta^{\theta}$, where $C$ depends on $T$, $\theta$, $E$, the structural parameters, and the prescribed initial norm. Consequently, $u(T,\cdot)=0$ on $\mathbb{R}^{d}$ whenever it vanishes on $E$.

In the classical case $k\equiv0$, Wang, Wang, Zhang, and Zhang established the equivalence of the three inequalities in \cite{GWW}. We also refer to that work for a detailed discussion from a control-theoretic perspective.

The spectral inequality is closely related to the Logvinenko--Sereda theorem, whose classical form was first established in \cite{LVN}.  More precisely,
 $E$ is $\gamma$-thick at scale $L$ for some $\gamma>0$ and $L>0$ if and only if $E$ satisfies the following inequality: for each $N>0$, there is a positive constant $C(d,E,N)$ such that
\begin{equation}
\int_{\mathbb{R}^{d}} |f(x)|^{2}\mathrm{d}x
\leq C(d,E,N) \int_{E} |f(x)|^{2}\mathrm{d}x
\quad \text{for each } f\in L^{2}(\mathbb{R}^{d}) \text{ with } \supp\widehat{f}\subset \overline{B(0,N)},
\end{equation}
where $\widehat{f}$ denotes the Fourier transform of the function $f$.

Kovrijkine \cite{OK} obtained quantitative bounds in terms of the thickness parameters and the bandwidth. Wang, Wang, Zhang, and Zhang \cite{GWW} used a spectral inequality of the form $e^{C(1+N)}$ in their characterization of observable sets for the classical heat equation. Ghobber and Jaming \cite{SGP} proved a Logvinenko--Sereda theorem for the Fourier--Bessel transform. Xu, Zhao, and Wei \cite{HX} subsequently treated the multidimensional Fourier--Bessel transform, with measure $\mathrm{d}\mu_\nu(x)=\prod_{i=1}^d x_i^{2\nu_i+1}\,\mathrm{d}x$ on $\mathbb{R}_+^d$. % [R8-MATH-P2-02] Made the thickness hypothesis explicit in the introductory spectral estimate.
In the Dunkl setting, we obtain the following spectral inequality whenever $E$ is thick with respect to $\mu_k$:
\begin{equation}
\int_{\mathbb{R}^{d}} |f(x)|^{2}\mathrm{d}\mu_k(x)
\leq e^{C(1+N)} \int_{E} |f(x)|^{2}\mathrm{d}\mu_k(x)
\end{equation}
for every $f\in L_k^2(\mathbb{R}^d)$ with $\supp(\mathcal{F}_k f)\subset\overline{B(0,N)}$,
where $C$ depends only on $d$, $E$, $k$, and $R_+$.
% [R8-LANG-P2-03] Polished the exposition without changing the argument.
The main difficulties are the loss of ordinary translation invariance and the nonconstant weight. A key step in our argument is a Bernstein inequality for ordinary derivatives of Dunkl-band-limited functions. For Fourier transform, referring to \cite{SNB},\cite{REA} or \cite{GEB}, when \(|\beta|=n\), we have 
\[
\|\partial^\beta f\|_{L^2(\mathbb{R}^d)}\leq C N^n\|f\|_{L^2(\mathbb{R}^d)}.
\]
However, due to the influence of singularity variations in the weighted space, we separate the region near the hyperplane from the rest for functions with finite Dunkl spectral bandwidth, and obtain Bernstein inequalities of the following form.

\begin{theorem}\label{TBNST}
Let $N>0$ and $f\in L_k^2(\mathbb{R}^d)$ satisfy $\supp(\mathcal{F}_k f)\subset\overline{B(0,N)}$. For every multi-index $\beta\in\mathbb{N}^d$ with $|\beta|=n$, we have
\begin{equation}
\|\partial^\beta f\|_{L_k^2(\mathbb{R}^d)}\leq C'(d,R_+,k)(C(d,R_+,k)N)^n\|f\|_{L_k^2(\mathbb{R}^d)}.
\end{equation}
\end{theorem}
 
For suitable examples, the constant $C(d,R_+,k)$ may be close to $1$, with a corresponding increase in $C'(d,R_+,k)$.

When all $k$ are restricted to positive integers, the Logvinenko--Sereda theorem for the Dunkl transform can be obtained by applying the Logvinenko--Sereda theorem and Paley--Wiener theorem of the Fourier transform to  $g = f \prod_{\alpha\in R_+}|\langle x,\alpha \rangle|^{k(\alpha)}$, where all $k$ are positive integers preserving analyticity. This paper establishes that the Logvinenko--Sereda theorem for the Dunkl transform still holds for general $k>0$.

This estimate leads to the following quantitative Logvinenko--Sereda theorem.
\begin{theorem}
\label{thm:main}
Let $E \subset \mathbb{R}^d$ be a measurable set. The following are equivalent:
\begin{enumerate}
\item[(i)] $E$ is relatively dense with respect to $\mu_k$.
\item[(ii)] There exists a constant $C = C(E,d,R_+,k)>0$ such that for every $N>0$ and every $f \in L_k^2(\mathbb{R}^d)$ with $\supp(\mathcal{F}_k f) \subset \overline{B(0,N)}$,
\[
\| f \|_{L_k^2(\mathbb{R}^d)}^2 \leq e^{C(1+N)} \| f \|_{L_k^2(E)}^2.
\]
\end{enumerate}
\end{theorem}
% [R8-LANG-P2-05] Polished the exposition without changing the argument.
Together with semigroup arguments, the spectral inequality yields the following characterization, as an application to control theory and characterization of unique continuation.

\begin{theorem}\label{TOHS}
Let $E \subset \mathbb{R}^d$ be a measurable set. Then the following statements are equivalent:
\begin{enumerate}
\item[(i)] The set $E$ is $\gamma$-thick at scale $L$ for some $\gamma > 0$ and $L > 0$ with respect to the Dunkl measure $\mu_k$.
\item[(ii)] The set $E$ satisfies the spectral inequality \eqref{ESPEC}.
\item[(iii)] The set $E$ satisfies the H\"older-type interpolation inequality \eqref{EHOLDER}.
\item[(iv)] The set $E$ satisfies the observability inequality \eqref{EOBS}.
\end{enumerate}
\end{theorem}

The main contributions of this paper are threefold. First, the Bernstein estimate for ordinary derivatives under the restriction of the Dunkl spectral support. Second, a local propagation small quantity that avoids weight degeneracy while preserving the linear bandwidth exponent. Third, the characterization of observable sets obtained from spectral inequalities and heat kernel localization.

% [R8-LANG-011] Condensed the paper outline.
Section~2 collects the required preliminaries. Section~3 proves the equivalence of the two notions of thickness. Sections~4 and~5 establish, respectively, the necessity and sufficiency in the Logvinenko--Sereda theorem. Section~6 completes the proof of Theorem~\ref{TOHS}.

\section{Notation and Preliminaries}

% [R8-LANG-P2-06] Polished the exposition without changing the argument.
We recall the definitions and properties of Dunkl analysis used below; see \cite{MFE1,MR1,ST1,JDA1} for further background.

\subsection{Dunkl theory}

Let $R \subset \mathbb{R}^d \setminus \{0\}$ be a normalized root system, and let $R_+$ be a fixed positive subsystem. For $\alpha \in R$, the reflection $\sigma_\alpha$ in the hyperplane orthogonal to $\alpha$ is given by
\[
\sigma_\alpha(x) = x - 2\frac{\langle \alpha,x\rangle}{|\alpha|^2}\alpha.
\]
The finite group generated by these reflections is denoted by $G$. Throughout this paper, the multiplicity function $k:R\to(0,\infty)$ is $G$-invariant. We set $k_\alpha = k(\alpha)$ for $\alpha \in R$.

The weight function (or Dunkl weight) is
\[
w(x) := \prod_{\alpha \in R_+} |\langle \alpha,x\rangle|^{2k_\alpha}, \qquad x \in \mathbb{R}^d.
\]
The associated measure is $\mathrm{d}\mu_k(x) = w(x)\,\mathrm{d}x$. We set $\gamma_k:=\sum_{\alpha\in R_+}k_\alpha$ and denote the homogeneous dimension by $d_k:=d+2\gamma_k$. Then
\[
\mu_k(B(tx, tr)) = t^{d_k} \mu_k(B(x, r)) \quad \text{for all } x\in\mathbb{R}^d,\ t,r>0,
\]

By \cite{JDA2} or \cite{JDA3}, there is a constant $C>0$ such that for all $x\in\mathbb{R}^d$ and $r>0$,
\begin{equation}\label{EWBOUND}
C^{-1} \mu_k(B(x,r))
% [R8-NOTATION-G002] Unified multiplicity/derivative/space/index notation or corrected a typographical inconsistency.
\le r^{d} \prod_{\alpha\in R_+} \big(|\langle x,\alpha\rangle| + r\big)^{2k_\alpha}
\le C \mu_k(B(x,r)),
\end{equation}

For fixed $y\in\mathbb{R}^d$, the Dunkl kernel $E_k(x,y)$ is the unique analytic solution to the system
\[
T_j f = y_j f,\quad j=1,\dots,d,\qquad f(0)=1,
\]
where $T_j$ are the Dunkl operators (see \cite{CFD1}),
\begin{equation}\label{ETJ}
T_jf(x)=\partial_j f(x)+\sum_{\alpha \in R}\frac{k_\alpha}{2}\langle \alpha,e_j\rangle \frac{f(x)-f(\sigma_\alpha(x))}{\langle \alpha,x\rangle}
\end{equation}
We denote the reflection correction in the $j$th Dunkl operator by $\mathcal{R}_j$:
\begin{equation}\label{ERH2}
\mathcal{R}_jf(x)=T_jf(x)-\partial_j f(x)=\sum_{\alpha \in R}\frac{k_\alpha}{2}\langle \alpha,e_j\rangle \frac{f(x)-f(\sigma_\alpha(x))}{\langle \alpha,x\rangle}.
\end{equation}
 The kernel $E_k(x,y)$ extends holomorphically to $\mathbb{C}^d \times \mathbb{C}^d$ and satisfies the following fundamental properties:
\begin{enumerate}
\item[(i)] $E_k(x,y) = E_k(y,x)$ for all $x,y \in \mathbb{R}^d$.
\item[(ii)] $E_k(tx,y) = E_k(x,ty)$ for all $t \in \mathbb{C}$.
\item[(iii)] $|E_k(ix,y)| \leq 1$ for all $x,y \in \mathbb{R}^d$.

\end{enumerate}

The Dunkl transform of $f \in L_k^1(\mathbb{R}^d)$ is defined by
\[
\mathcal{F}_k f(\xi) := c_k \int_{\mathbb{R}^d} f(x) E_k(-i\xi, x) \, \mathrm{d}\mu_k(x),
\]
where
\[
c_k=\left(\int_{\mathbb{R}^d}e^{-|x|^2/2}\,\mathrm{d}\mu_k(x)\right)^{-1}.
\]
With this convention, $\mathcal{F}_k$ is an isometry on $L_k^2(\mathbb{R}^d)$. The inverse transform is given by
\[
\mathcal{F}_k^{-1} f(x) = c_k \int_{\mathbb{R}^d} f(\xi) E_k(i x, \xi) \, \mathrm{d}\mu_k(\xi) = \mathcal{F}_k f(-x).
\]
% [R8-REF-012] Cited the precise source for the spherical Paley--Wiener theorem.
We recall the Plancherel identity from \cite{MFE1} and the spherical Paley--Wiener theorem from \cite[Theorem~4.10]{MFE2}.

 \emph{Plancherel identity.} $\| \mathcal{F}_k f \|_{L_k^2} = \| f \|_{L_k^2}$ for all $f \in L_k^2(\mathbb{R}^d)$.
\begin{lemma}\label{PWSDK}
% [R8-MATH-013] Corrected the misquoted Paley--Wiener statement: Schwartz f and decay in z, not N.
(Paley--Wiener) Let $N>0$ and $f\in\mathcal{S}(\mathbb{R}^d)$. Then $\supp(\mathcal{F}_k f)\subset\overline{B(0,N)}$ if and only if $f$ extends to an entire function on $\mathbb{C}^d$ such that, for every $m\in\mathbb{N}$, there is a constant $C_m>0$ satisfying
\[
|f(z)|\le C_m(1+|z|)^{-m}e^{N|\operatorname{Im}z|},\qquad z\in\mathbb{C}^d.
\]
\end{lemma}

% [R8-REF-014] Avoided attributing an unspecified distributional Paley--Wiener theorem to the convolution-equations reference.
Further Paley--Wiener results are given in \cite{MFE2}, and distributional aspects of Dunkl analysis are treated in \cite{KT1}. For the classical Fourier-transform statement below, see \cite{LH1}.
\begin{lemma}\label{LPYC}
Let $K\subset\mathbb{R}^d$ be compact and convex, with support function $H(\eta)=\sup_{x\in K}\langle x,\eta\rangle$. If $u\in C_c^\infty(\mathbb{R}^d)$ and $\supp u\subset K$, then for every $m\in\mathbb{N}$ there exists $C_m>0$ such that
\[
|\widehat{u}(z)|\leq C_m(1+|z|)^{-m}e^{H(\operatorname{Im} z)}, \qquad z\in \mathbb{C}^d.
\]
\end{lemma}

\subsection{Dunkl translation}\label{subsec:translation}

The generalized translation operator $\tau_y$ is defined on $L_k^2(\mathbb{R}^d)$ by
\[
\mathcal{F}_k(\tau_y f)(x) = E_k(y, -i x) \mathcal{F}_k f(x), \qquad x \in \mathbb{R}^d.
\]
Equivalently, for $f \in L_k^1(\mathbb{R}^d) \cap L_k^2(\mathbb{R}^d)$ with $\mathcal{F}_k f \in L_k^1(\mathbb{R}^d)$,
\[
\tau_y f(x) = c_k \int_{\mathbb{R}^d} E_k(i x, \xi) E_k(-i y, \xi) \mathcal{F}_k f(\xi) \, \mathrm{d}\mu_k(\xi).
\]
For a radial function $f(x)=\widetilde f(|x|)$, R\"osler's positive representation \cite{MR3}, expressed in the present negative-translation convention, gives
\[
\tau_y f(x)=\int_{\operatorname{conv}(Gx)}
\widetilde f\!\left(\sqrt{|x|^2+|y|^2-2\langle\eta,y\rangle}\right)\,\mathrm{d}\nu_x(\eta),
\]
where $\nu_x$ is a probability measure supported in $\operatorname{conv}(Gx)$. In particular, radial nonnegative functions have nonnegative translations. Notice that at $k=0$ our convention is $\tau_y f(x)=f(x-y)$; some references use the opposite sign.

The support property needed below is given by \cite[Theorem~1.7]{JDA2}:

If $f \in L_k^2(\mathbb{R}^d)$ and $\supp f \subset \overline{B(0,r)}$, then for every $x \in \mathbb{R}^d$,
\[
\supp \tau_x f \subset \bigcup_{\sigma \in G} \overline{B(\sigma(x), r)}.
\]

% [R8-LANG-P2-07] Polished the exposition without changing the argument.
The next estimate is the radial Schwartz specialization of \cite[Theorem~4.1 and Remark~4.2]{JDA3}, written in our translation convention.
\begin{lemma}\label{LPY}
Assume that $k_\alpha>0$ for every root. Let $g\in\mathcal{S}(\mathbb{R}^d)$ be radial and set $g_t(x)=t^{-d_k}g(x/t)$. For every $M>0$ there exists $C_{g,M}>0$ such that
\[
|\tau_xg_t(y)|\le\frac{C_{g,M}}{\mu_k(B(x,t))}
\left(1+\frac{|x-y|}{t}\right)^{-1}
\left(1+\frac{d_G(x,y)}t\right)^{-M},\qquad t>0,
\]
% [R8-NOTATION-INLINE-01] Standardized inline mathematics to dollar delimiters.
where $d_G(x,y)$ denotes the $G$-invariant distance,
\begin{equation}\label{EDD}
d_G(x,y)=\min_{\sigma\in G}|x-\sigma y|.
\end{equation}
The constants may depend on the relevant Schwartz seminorms of $g$.
\end{lemma}
Taking $t=1$ gives the following estimate:
\begin{lemma}\label{lem:translation-estimate}
Under the same multiplicity assumption, for every radial $\phi\in\mathcal{S}(\mathbb{R}^d)$ and $M>0$,
\[
|\tau_x\phi(y)|\le\frac{C_{\phi,M}}{\mu_k(B(x,1))}
(1+|x-y|)^{-1}(1+d_G(x,y))^{-M}.
\]
\end{lemma}

For radial functions, Dunkl translation is contractive on $L_k^p$ for $1\le p\le\infty$; see \cite{ST1}:
\begin{lemma}
\label{lem:radial-translation}
If $f \in L_k^p(\mathbb{R}^d)$ is radial, then $\tau_y f \in L_k^p(\mathbb{R}^d)$ for every $y \in \mathbb{R}^d$, and
\[
\| \tau_y f \|_{L_k^p} \leq \| f \|_{L_k^p}.
\]
\end{lemma}

\subsection{Dunkl convolution}

The Dunkl convolution of two functions $f,g \in L_k^2(\mathbb{R}^d)$ is defined by
\[
f *_k g(x) := \int_{\mathbb{R}^d} f(y) \tau_x g^\vee(y) \, \mathrm{d}\mu_k(y), \qquad g^\vee(y) = g(-y).
\]
Equivalently, using the Dunkl transform,
\[
\mathcal{F}_k(f *_k g) = c_k^{-1}(\mathcal{F}_k f)(\mathcal{F}_k g).
\]
The convolution is commutative: $f *_k g = g *_k f$.

% [R8-LANG-P2-08] Polished the exposition without changing the argument.
When at least one factor is radial, the following Young inequality holds; see \cite{ST1}.
\begin{lemma}
\label{lem:young}
Let $1 \leq p,q,r \leq \infty$ satisfy $1/p + 1/q = 1 + 1/r$. If $f \in L_k^p(\mathbb{R}^d)$ and $g \in L_k^q(\mathbb{R}^d)$, and one of them is radial, then $f *_k g \in L_k^r(\mathbb{R}^d)$ and
\[
\| f *_k g \|_{L_k^r} \leq \| f \|_{L_k^p} \| g \|_{L_k^q}.
\]
In particular, taking $p = \infty, q = 1, r = \infty$ gives
\[
\| f *_k g \|_{L^\infty} \leq \| f \|_{L^\infty} \| g \|_{L_k^1}.
\]

\end{lemma}

\subsection{Dunkl heat kernel}

The Dunkl Laplacian associated with $R$ and $k$ is the differential--difference operator
$\Delta_k=\sum_{j=1}^d T_j^2$, where $T_j$ is defined by \eqref{ETJ}.

% [R8-LANG-P2-09] Polished the exposition without changing the argument.
The heat semigroup $H_t=e^{t\Delta_k}$ introduced in Section~1 has the kernel representation
\[
H_t f(x)=\int_{\mathbb{R}^d} h_t(x,y) f(y)\,w(y)\mathrm{d}y,
\]
where the heat kernel
\begin{equation}\label{eq:2.25}
h_t(x,y)=c_k(2t)^{-d_k/2} E_k\!\left(\frac{x}{\sqrt{2t}},\frac{y}{\sqrt{2t}}\right) e^{-(|x|^2+|y|^2)/(4t)}
\end{equation}
is a $C^\infty$ function of all variables $x,y\in\mathbb{R}^d$, $t>0$, and satisfies
\begin{equation}\label{eq:2.26}
0<h_t(x,y)=h_t(y,x).
\end{equation}

% [R8-REF-047] Recorded the two standard semigroup identities explicitly used later.
We also use the mass-conservation and multiplier identities (see \cite{MR2}):
\[
\int_{\mathbb{R}^d}h_t(x,y)\,\mathrm{d}\mu_k(y)=1,\qquad
\mathcal{F}_k(H_tf)(\xi)=e^{-t|\xi|^2}\mathcal{F}_k f(\xi).
\]

The closures of connected components of
\[
\{x\in\mathbb{R}^d:\langle x,\alpha\rangle\neq 0 \text{ for all } \alpha\in R\}
\]
are called (closed) \textit{Weyl chambers}. % [R8-LANG-P2-10] Polished the exposition without changing the argument.
The orbit distance $d_G(x,y)$ is defined in \eqref{EDD}.

% [R8-NOTATION-G003] Unified multiplicity/derivative/space/index notation or corrected a typographical inconsistency.
The upper and lower bounds for the Dunkl heat kernel established in \cite{JDA1} are recalled in Lemma~\ref{LBOUND} below.

For a finite sequence $\boldsymbol{\alpha}=(\alpha_1,\alpha_2,\dots,\alpha_m)$ of elements of $R_+$, $x,y\in\mathbb{R}^d$ and $t>0$, let
\begin{equation}\label{eq:1.3}
\ell(\boldsymbol{\alpha}):=m
\end{equation}
be the length of $\boldsymbol{\alpha}$,
\begin{equation}\label{eq:1.4}
\sigma_{\boldsymbol{\alpha}}:=\sigma_{\alpha_m}\circ\sigma_{\alpha_{m-1}}\circ\dots\circ\sigma_{\alpha_1},
\end{equation}

For $m\ge1$, define
\[
\rho_{\boldsymbol{\alpha}}(x,y,t):=\prod_{j=0}^{m-1}\left(1+\frac{|x-\sigma_{\alpha_j}\circ\cdots\circ\sigma_{\alpha_1}(y)|}{\sqrt{t}}\right)^{-2},
\]
where the prefix with $j=0$ is the identity map.

For $x,y\in\mathbb{R}^d$, let $n(x,y)=0$ if $d_G(x,y)=|x-y|$ and
\[
n(x,y)=\min\bigl\{m\in\mathbb{N}^+: d_G(x,y)=\bigl|x-\sigma_{\alpha_m}\circ\cdots\circ\sigma_{\alpha_2}\circ\sigma_{\alpha_1}(y)\bigr|,\ \alpha_j\in R\bigr\}.
\]

We say that a finite sequence $\boldsymbol{\alpha}=(\alpha_1,\alpha_2,\dots,\alpha_m)$ of positive roots is \textit{admissible for the pair} $(x,y)\in\mathbb{R}^d\times\mathbb{R}^d$ if $n(x,\sigma_{\boldsymbol{\alpha}}(y))=0$. In other words, the composition
\[
\sigma_{\boldsymbol{\alpha}}=\sigma_{\alpha_m}\circ\sigma_{\alpha_{m-1}}\circ\cdots\circ\sigma_{\alpha_1}
\]
of the reflections $\sigma_{\alpha_j}$ maps $y$ to a Weyl chamber that also contains $x$.

The set of all admissible sequences $\boldsymbol{\alpha}$ for the pair $(x,y)$ will be denoted by $\mathcal{A}(x,y)$. Note that if $n(x,y)=0$, then $\boldsymbol{\alpha}=\emptyset\in\mathcal{A}(x,y)$.

Define
\begin{equation}\label{eq:1.7}
\Lambda(x,y,t):=\sum_{\boldsymbol{\alpha}\in\mathcal{A}(x,y),\,\ell(\boldsymbol{\alpha})\le 2|G|}\rho_{\boldsymbol{\alpha}}(x,y,t).
\end{equation}

For the empty sequence, set $\rho_{\emptyset}(x,y,t)=1$. The following upper and lower bounds were proved in \cite{JDA1} for strictly positive multiplicities.
\begin{lemma}\label{LBOUND}
Assume that $k_\alpha>0$ for all roots, $0<c_u<1/4$, and $c_l>1/4$. Then there are constants $C_u, C_l > 0$ such that for all $x, y \in \mathbb{R}^d$ and $t > 0$ we have
\begin{equation}\label{E2.1}
C_l \mu_k\big(B(x, \sqrt{t})\big)^{-1} e^{-c_l \frac{d_G(x,y)^2}{t}} \Lambda(x, y, t) \leq h_t(x, y),
\end{equation}
\begin{equation}\label{E2.2}
h_t(x, y) \leq C_u \mu_k\big(B(x, \sqrt{t})\big)^{-1} e^{-c_u \frac{d_G(x,y)^2}{t}} \Lambda(x, y, t).
\end{equation}
\end{lemma}

\subsection{Auxiliary lemmas}

% [R8-LANG-P2-11] Polished the exposition without changing the argument.
We use the polynomial sublevel-set estimate of \cite{ACJ}.

\begin{lemma}[Carbery--Wright, Theorem 2]
\label{lem:carbery-wright}
Let $X$ be a Banach space and let $p:\mathbb{R}^d \to X$ be a polynomial of degree at most $m$, where $m\ge1$. Let $K \subset \mathbb{R}^d$ be a convex body with $|K|=1$. Define $p^\sharp(x) := \|p(x)\|^{1/m}$. Then for any $0 \le q \le \infty$ and any $\delta > 0$, there exists an absolute constant $C>0$, independent of $p,m,K,d,q,X$, such that
\[
\|p^\sharp\|_{L^q(K)} \, \delta^{-1} \, \bigl|\{ x \in K : p^\sharp(x) \le \delta \}\bigr|
\le C\, d\, \bigl(d B(d,q+1)\bigr)^{1/q},
\]
where $B$ denotes the classical Beta function, with the usual interpretations for $q=0$ and $q=\infty$. In particular, for $q=\infty$,
\[
\|p^\sharp\|_{L^\infty(K)} \, \delta^{-1} \, \bigl|\{ x \in K : p^\sharp(x) \le \delta \}\bigr| \le C\, d.
\]
\end{lemma}

\begin{remark}\label{R29}
In the scalar-valued setting $X=\mathbb{C}$, taking $m=1$ and $p(x)=\langle x,\alpha\rangle$ for $\alpha\in R_+$, we have $p^\sharp(x)=|p(x)|$. The $q=\infty$ case then reads
\[
\sup_{x\in K} |\langle x,\alpha\rangle| \, \delta^{-1} \, \bigl|\{ x \in K : |\langle x,\alpha\rangle| \le \delta \}\bigr| \le C\, d.
\]
This is the estimate employed in the proof of Theorem~\ref{TEV} to control the measure of the set where the linear forms are small.
\end{remark}

% [R8-LANG-P2-12] Polished the exposition without changing the argument.
We also recall the Remez inequality for polynomials; see \cite{EJR}.

Let $f$ be a polynomial of degree $m$, $I\subset\mathbb {R}$ an interval, and $E\subset I$ measurable. Then
\begin{equation}
\|f\|_I \le \left(\frac{4|I|}{|E|}\right)^m\|f\|_E,
\end{equation}
where $\|f\|_I=\sup_{x\in I} |f(x)|$.

 For analytic functions, we use the following result from \cite[Lemma~1]{OK} in the proof of the sufficiency part of the Logvinenko--Sereda theorem.

\begin{lemma}\label{L1}
Let $\Phi$ be an analytic function on $D_5(0)$, the open disk in $\mathbb{C}$ centered at the origin with radius $5$. Let $I$ be an interval of length $1$ satisfying $0 \in I$, and let $\hat{E} \subset I$ be a subset of positive Lebesgue measure. If $|\Phi(0)| \geq 1$ and $M = \max_{|z|\leq 4} |\Phi(z)|$, then there exists an absolute constant $C > 0$ such that
\begin{equation}
\sup_{x\in I} |\Phi(x)| \leq \left(\frac{C}{|\hat{E}|}\right)^{\frac{\ln M}{\ln 2}} \sup_{x\in \hat{E}} |\Phi(x)|.
\end{equation}
\end{lemma}
 Related quantitative estimates for quasianalytic functions are given in \cite{FNM}.
\section{Equivalence of Thick Sets}
\label{sec:thick-equivalence}

In this section, we prove the equivalence between thick sets with respect to the Dunkl measure $\mu_k$ and thick sets with respect to the Lebesgue measure. % [R8-LANG-P2-13] Polished the exposition without changing the argument.
This allows us to use Euclidean geometric estimates in the weighted setting.

% [R8-LANG-039] Improved precision, flow, or concision in the proof exposition.
We first show that every unit cube contains a point where the weight is uniformly bounded below. For each function $\ell_\alpha(x)=|\langle \alpha,x\rangle|$, the supremum of $\ell_\alpha$ over any unit cube, namely $\sup_{x\in K} |\langle x,\alpha\rangle|$, admits a uniform positive lower bound.

\begin{lemma}[Uniform Nonvanishing in Unit Cubes]
\label{lem:lower-bound}
% [R8-MATH-015] Made the simultaneous nonvanishing conclusion explicit, as already proved in Lemma 3.1.
There exist a constant $C_0=C_0(d,R_+,k)>0$ and a unit vector $u\in\mathbb{S}^{d-1}$ such that, for every $x\in\mathbb{R}^d$, one can find $t\in[-1/2,1/2]$ with
\[
|\langle\alpha,x+tu\rangle|\ge C_0\qquad\text{for every }\alpha\in R_+.
\]
\end{lemma}

\begin{proof}
We first prove that there exists a unit vector $u \in \mathbb{S}^{d-1}$ such that $\langle \alpha, u \rangle \neq 0$ for all $\alpha \in R_+$. If no such vector existed, the unit sphere would be entirely covered by the finite union of hyperplanes $\{\langle \alpha, x \rangle = 0\}_{\alpha \in R_+}$, which is impossible.

Fix such a unit vector $u$ and define
\[
\gamma := \min_{\alpha \in R_+} |\langle \alpha, u \rangle| > 0, \qquad C_0 := \frac{\gamma}{4|R_+|}.
\]
For fixed $x \in \mathbb{R}^d$ and each root $\alpha \in R_+$, define the scalar function
\[
f_\alpha(t) = |\langle \alpha, x + tu \rangle| = \bigl|\langle \alpha, x \rangle + t\langle \alpha, u \rangle\bigr|.
\]
The set $\{ t \in [-1/2,1/2] : f_\alpha(t) < C_0 \}$ is an interval of length at most
\[
\frac{2C_0}{|\langle \alpha, u \rangle|} \leq \frac{1}{2|R_+|}.
\]
Taking the union of these intervals over all $\alpha \in R_+$, the total Lebesgue measure is at most $1/2$. Consequently, there exists $t \in [-1/2, 1/2]$ such that
\[
|\langle \alpha, x + tu \rangle| \geq C_0 \quad \text{for all } \alpha \in R_+.
\]
This completes the proof.
\end{proof}

\begin{remark}
In particular, every ball of radius $1/2$ contains a point where the weight is bounded below by a uniform positive constant; this is not a lower bound at every point of the ball. Precisely,
\[
w(x + tu) = \prod_{\alpha \in R_+} \bigl|\langle \alpha, x + tu \rangle\bigr|^{2k_\alpha} \geq C_0^{2\sum_{\alpha\in R_+} k_\alpha} =: C_0' > 0,
\]
% [R8-NOTATION-G004] Unified multiplicity/derivative/space/index notation or corrected a typographical inconsistency.
where $C_0'=C_0'(d,R_+,k)$. Lemma~\ref{lem:lower-bound} implies that the supremum of $w(x)$ over any unit cube is uniformly bounded below by a positive constant. However, this result does not yield a lower bound for the infimum of $w(x)$, since $w(x)$ vanishes on the reflection hyperplanes $\langle \alpha, x \rangle = 0$.
\end{remark}

We now prove the equivalence of relative density stated in Theorem~\ref{TEV}.

\begin{proof}
% [R8-LANG-016] Condensed the proof roadmap.
We first prove the result for unit cubes and then use homogeneity to treat an arbitrary side length.

\medskip
\noindent\textbf{Step 1: The case $L=1$.}

Suppose first that $E$ is $(\gamma,1)$-relatively dense with respect to Lebesgue measure, i.e., for every unit cube $Q$,
\[
|E \cap Q| \ge \gamma |Q| = \gamma.
\]
We claim that there exists a constant $\gamma_1>0$, depending only on $\gamma,d,R_+,k$, such that
\[
\mu_k(E \cap Q) \ge \gamma_1 \mu_k(Q)
\]
holds for all unit cubes $Q$.

% [R8-LANG-017] Simplified the definition and clarified use of the convex-body lemma on the closed cube.
Fix a unit cube $Q$ and set $\ell_\alpha(x)=|\langle x,\alpha\rangle|$ for each $\alpha\in R_+$. Applying the $q=\infty$ case of Lemma~\ref{lem:carbery-wright} to $p(x)=\langle x,\alpha\rangle$ on $\overline Q$ (whose boundary has measure zero), we obtain
\begin{equation}\label{E6}
\sup_{x\in Q} \ell_\alpha(x) \cdot \delta^{-1} \cdot \bigl|\{ x\in Q : \ell_\alpha(x) \le \delta \}\bigr| \le C d,
\end{equation}
where $C$ is an absolute constant.
% [R8-NOTATION-G005] Unified multiplicity/derivative/space/index notation or corrected a typographical inconsistency.
By Lemma~\ref{lem:lower-bound}, there exists $C_0>0$ independent of $Q$ and $\alpha$ satisfying $\sup_{x\in Q} \ell_\alpha(x) \ge C_0$. Rearranging the inequality yields
\[
\bigl|\{ x\in Q : \ell_\alpha(x) \le \delta \}\bigr| \le \frac{C d}{C_0} \delta.
\]
Choose $\delta>0$ sufficiently small that
\[
\frac{C d|R_+|}{C_0} \delta \le \frac{\gamma}{10}.
\]
Remove thin neighborhoods of the hyperplanes by defining
\[
Q_2 := Q \setminus \bigcup_{\alpha \in R_+} \bigl\{ x\in Q : \bigl|\langle x,\alpha\rangle\bigr| \le \delta \bigr\}.
\]
By construction,
\[
|Q_2| \ge 1 - \frac{\gamma}{10} = \left(1 - \frac{\gamma}{10}\right)|Q|,
\]
and every $x\in Q_2$ satisfies $\bigl|\langle x,\alpha\rangle\bigr| \ge \delta$ for all $\alpha\in R_+$. On $Q_2$, the weight $w(x)$ has bounded oscillation; note that the choice of $\delta$ depends only on the constants $C$, $d$, $|R_+|$, $\gamma$, and $C_0$.
\begin{equation}\label{E1}
\frac{\sup_{x\in Q} w(x)}{\inf_{x\in Q_2} w(x)}
\le \prod_{\alpha\in R_+} \left(1 + \frac{|\alpha|\sqrt{d}}{\delta}\right)^{2k_\alpha}
=: C_1,
\end{equation}
where $C_1$ depends only on $R_+, k, \gamma, d$. Since $|E\cap Q| \ge \gamma$, we obtain the following lower bound for the Lebesgue mass contained in $Q_2$:
\[
|E\cap Q_2| \ge |E\cap Q| - |Q\setminus Q_2| \ge \gamma - \frac{\gamma}{10} = \frac{9}{10}\gamma.
\]
We now bound the weighted measure of $E\cap Q$ from below:
\[
\mu_k(E\cap Q) \ge \mu_k(E\cap Q_2)
= \int_{E\cap Q_2} w(x)\,\mathrm{d}x
\ge \inf_{x\in Q_2} w(x) \cdot |E\cap Q_2|
\ge \frac{9}{10}\gamma \inf_{x\in Q_2} w(x).
\]
On the other hand,
\[
\mu_k(Q) = \int_Q w(x)\,\mathrm{d}x \le \sup_{x\in Q} w(x) \le C_1 \inf_{x\in Q_2} w(x).
\]
Combining the two inequalities yields
\[
\mu_k(E\cap Q) \ge \frac{9}{10C_1}\gamma \mu_k(Q).
\]
Setting $\gamma_1 := \frac{9}{10C_1}\gamma$ completes the forward implication for unit cubes with $L=1$.

Conversely, assume $E$ is $(\gamma,1)$-relatively dense with respect to $\mu_k$, i.e.,
\[
\mu_k(E\cap Q) \ge \gamma \mu_k(Q)
\]
holds for every unit cube $Q$. We derive a lower bound for the Lebesgue measure:
\[
|E\cap Q| \ge \frac{\mu_k(E\cap Q)}{\sup_{x\in E\cap Q} w(x)}
\ge \frac{\gamma \mu_k(Q)}{\sup_{x\in E\cap Q} w(x)}
\ge \frac{\gamma \mu_k(Q)}{\sup_{x\in   Q} w(x)}.
\]

% [R8-LANG-018] Clarified that the reverse implication repeats the construction rather than reusing a previous hypothesis.
Construct $Q_2$ with this density parameter $\gamma$ as above. Then
\begin{align*}
|E\cap Q| &\ge \frac{\gamma}{\sup_{x\in Q} w(x)} \mu_k(Q) \\
&\ge \frac{\gamma}{\sup_{x\in Q} w(x)} \inf_{x\in Q_2} w(x) \, |Q_2| \\
&\ge \gamma \cdot \prod_{\alpha\in R_+} \left( \frac{\delta}{\delta+|\alpha|\sqrt{d}} \right)^{2k_\alpha} \cdot \left(1-\frac{\gamma}{10}\right) \cdot |Q|.
\end{align*}

Define $\gamma_1' := \gamma \cdot \prod_{\alpha\in R_+} \left( \frac{\delta}{\delta+|\alpha|\sqrt{d}} \right)^{2k_\alpha} \cdot \left(1-\frac{\gamma}{10}\right)$. Then $|E\cap Q|\ge \gamma_1'$ for all unit cubes $Q$, which proves the reverse direction for $L=1$.

\medskip
\noindent\textbf{Step 2: Extension to arbitrary $L$ via scaling.}

Let $L>0$ be arbitrary. Introduce the scaled set
\[
E_L := \{ y \in \mathbb{R}^d : L y \in E \}.
\]
For any cube $Q$ of side length $L$, write $Q$ as $Q = L Q_1$, where $Q_1$ is a unit cube.

For the forward implication, suppose that $E$ is $(\gamma,L)$-relatively dense with respect to Lebesgue measure, so $|E\cap Q| \ge \gamma |Q| = \gamma L^d$. A linear change of variables gives
\[
|E_L \cap Q_1| = \frac{1}{L^d} |E \cap Q| \ge \gamma,
\]
so $E_L$ is $(\gamma,1)$-relatively dense with respect to Lebesgue measure. By Step 1, there exists $\gamma_1>0$ (depending only on $\gamma,d,R_+,k$) such that
\[
\mu_k(E_L \cap Q_1) \ge \gamma_1 \mu_k(Q_1).
\]
Recall the homogeneity property of $\mu_k$: for any measurable set $A$ and $\lambda>0$, $\mu_k(\lambda A) = \lambda^{d_k} \mu_k(A)$, with $\gamma_k = \sum_{\alpha\in R_+} k_\alpha$; see \cite{MR1}. Applying this identity with $\lambda=1/L$,
\[
\mu_k(E_L \cap Q_1) = \mu_k\left( \frac{1}{L} (E\cap Q) \right)
% [R8-NOTATION-G006] Unified multiplicity/derivative/space/index notation or corrected a typographical inconsistency.
= L^{-d_k} \mu_k(E\cap Q),
\]
and
\[
% [R8-NOTATION-G007] Unified multiplicity/derivative/space/index notation or corrected a typographical inconsistency.
\mu_k(Q_1) = L^{-d_k} \mu_k(Q).
\]
Cancelling the common scaling factor yields
\[
\mu_k(E\cap Q) \ge \gamma_1 \mu_k(Q).
\]
Hence $E$ is $(\gamma_1, L)$-relatively dense with respect to $\mu_k$.

Conversely, suppose $E$ is $(\gamma,L)$-relatively dense with respect to $\mu_k$, i.e., $\mu_k(E\cap Q) \ge \gamma \mu_k(Q)$. By the same homogeneity relation,
\[
\mu_k(E_L \cap Q_1) \ge \gamma \mu_k(Q_1),
\]
so $E_L$ satisfies the $(\gamma,1)$-density condition for $\mu_k$. From Step 1, there exists $\gamma_1'>0$ such that
\[
|E_L \cap Q_1| \ge \gamma_1'.
\]
Scaling back to the original cube $Q=L Q_1$,
\[
|E\cap Q| = L^d |E_L \cap Q_1| \ge \gamma_1' L^d = \gamma_1' |Q|.
\]
Thus $E$ is $(\gamma_1', L)$-relatively dense with respect to Lebesgue measure.

This completes the proof.
\end{proof}
\begin{remark} \label{RR3}
% [R8-LANG-019] Distinguished invariance of the notion from invariance of the thickness parameters.
The orientation of the cubes does not affect the qualitative notion of thickness, although the parameters may change. More specifically, if a measurable subset $E\subset \mathbb{R}^d$ is a $(\gamma, L)$-relatively dense subset, and $Q'$ denotes an arbitrary rotation of $Q$ about its center, since $x+LQ\subset x+L\sqrt{d}Q'$, we have
\[
|E\cap (x+L\sqrt{d} Q')|\geq \gamma' |x+L\sqrt{d}Q'|
\]
 where $\gamma'$ only depends on $d,\gamma$.

\end{remark}
\section{Necessity of the Logvinenko--Sereda Theorem}
\label{sec:necessity}
 
% [R8-LANG-020] Removed the overly broad claim that explicit cube integration is impossible for Dunkl weights.
Our argument is motivated by the Fourier--Bessel setting in \cite{HX}. There, the product weight $\mathrm{d}\mu_\nu(x)=\prod_{i=1}^d x_i^{2\nu_i+1}\,\mathrm{d}x$ allows explicit integration over coordinate cubes. For a general Dunkl weight, we instead use a uniform comparison of the measures of concentric cubes.

\begin{lemma}\label{L4001}
Let
\[
Q=\biggl(-\frac12,\frac12\biggr)^d,\quad Q_0=x_0+LQ,\quad Q_1=x_0+(L+M)Q,
\]
where $M\ge 0$ is fixed. Then there exists a constant $C>0$ depending only on $d,R,k$ such that for all $x_0\in\mathbb{R}^d$ and $L>0$,
\begin{equation}\label{E4001}
0\le 1-\frac{\mu_k(Q_0)}{\mu_k(Q_1)}\le C\,\frac{M}{L+M}.
\end{equation}

Therefore,
\begin{equation}\label{E4002}
\lim_{L\to\infty}\sup_{x_0\in\mathbb{R}^d}\left|\frac{\mu_k(x_0+LQ)}{\mu_k(x_0+(L+M)Q)}-1\right|=0.
\end{equation}

In particular, for any $\delta>0$, there exists $K>0$ such that for all $L>K$, uniformly in $x_0\in\mathbb{R}^d$,
\[
\frac{\mu_k(Q_0)}{\mu_k(Q_1)}>1-\delta.
\]

\end{lemma}
\begin{proof}
% [R8-LANG-041] Improved precision, flow, or concision in the proof exposition.
The cases $M=0$ and $R_+=\emptyset$ are immediate. We therefore assume $M>0$ and $R_+\ne\emptyset$.
For Lebesgue measure, the volume ratio of two concentric cubes tends to $1$ as $L$ grows. To account for the weight, we show that the strip between the cubes carries only a small fraction of the total mass. The proof proceeds in three steps: we estimate the weight, control the strip, and then complete the argument.

\emph{Step 1. A one-dimensional weight estimate.}

 Fix $m\ge1$ and $\kappa\ge 0$. Let
\[
h(t)=A\prod_{j=1}^{r}|a_j t+b_j|^{\beta_j},
\]
where
\[
A\ge 0,\qquad r\le m,\qquad \beta_j\ge 0,\qquad \sum_{j=1}^{r}\beta_j\le \kappa.
\]
There is a constant depending only on $m$ and $\kappa$, which may be chosen as
\[
C_{m,\kappa}=2(1+4m)^{\kappa},
\]
such that for any finite interval $I\subset\mathbb{R}$ and any measurable set $E\subset I$, we have
\[
\int_{E} h(t)\,\mathrm{d}t \le C_{m,\kappa}\frac{|E|}{|I|}\int_{I} h(t)\,\mathrm{d}t.
\]

By an affine transformation, we may reduce the interval $I$ to $[0,1]$.
If a factor vanishes identically with a positive exponent, then $h\equiv0$ and there is nothing to prove. All nonzero constant factors corresponding to $a_j=0$ can be absorbed into $A$;
for $a_j\neq 0$, write
\[
|a_j t + b_j|^{\beta_j}=|a_j|^{\beta_j}|t-\rho_j|^{\beta_j},\qquad \rho_j=-\frac{b_j}{a_j}.
\]
Thus it suffices to consider
\[
h(t)=A_0\prod_{j=1}^{r}|t-\rho_j|^{\beta_j},\qquad t\in[0,1].
\]
Let
\[
\eta:=\frac{1}{4m},
\]
and define
\[
F:=[0,1]\setminus \bigcup_{j=1}^{r}(\rho_j-\eta,\rho_j+\eta).
\]
Since $r\le m$, we have
\begin{equation}
|F|\ge 1-2r\eta\ge 1-\frac12=\frac12.
\end{equation}

For any $t\in F$, we have
\[
|t-\rho_j|\ge \eta.
\]
Therefore, for any $s\in[0,1]$,
\[
\begin{aligned}
|s-\rho_j| &\le |s-t|+|t-\rho_j| \\
&\le 1+|t-\rho_j| \\
&\le (1+\eta^{-1})|t-\rho_j| \\
&=(1+4m)|t-\rho_j|.
\end{aligned}
\]
Hence
\[
h(s)\le (1+4m)^{\sum_j\beta_j}h(t)\le (1+4m)^{\kappa}h(t).
\]
Taking the supremum over $s\in[0,1]$, we obtain
\[
\|h\|_{L^{\infty}(0,1)}\le (1+4m)^{\kappa}h(t),\qquad t\in F.
\]
Integrating over $t\in F$, and using $|F|\ge 1/2$, we get
\[
\begin{aligned}
\int_{0}^{1}h(t)\,\mathrm{d}t &\ge \int_{F}h(t)\,\mathrm{d}t \\
&\ge \frac{|F|}{(1+4m)^{\kappa}}\|h\|_{L^{\infty}(0,1)} \\
&\ge \frac{1}{2(1+4m)^{\kappa}}\|h\|_{L^{\infty}(0,1)}.
\end{aligned}
\]
Therefore
\[
\|h\|_{L^{\infty}(0,1)}\le 2(1+4m)^{\kappa}\int_{0}^{1}h(t)\,\mathrm{d}t.
\]
Thus for any $E\subset [0,1]$,
\[
\int_{E} h(t)\,\mathrm{d}t \le |E|\,\|h\|_{L^{\infty}(0,1)}
\le 2(1+4m)^{\kappa}|E|\int_{0}^{1} h(t)\,\mathrm{d}t.
\]
Scaling back to the general interval $I$ via an affine transformation, we obtain
\[
\int_{E} h(t)\,\mathrm{d}t \le 2(1+4m)^{\kappa}\frac{|E|}{|I|}\int_{I} h(t)\,\mathrm{d}t.
\]
This proves the claim.

\emph{Step 2. A uniform estimate for the boundary strip.}

Fix
\[
a\in\mathbb{R}^d,\qquad \varepsilon>0,
\]
and define
\[
C_{\mathrm{in}}:=a+Q,\qquad C_{\mathrm{out}}:=a+(1+\varepsilon)Q.
\]
Since $Q=\bigl(-1/2,1/2\bigr)^d$, we have
\begin{equation}\label{EPL05}
C_{\mathrm{in}}\subset C_{\mathrm{out}}.
\end{equation}
For the $i$-th coordinate, let
\[
I_i:=\biggl(a_i-\frac{1+\varepsilon}{2},\,a_i+\frac{1+\varepsilon}{2}\biggr),
\]
and
\[
J_i:=\biggl(a_i-\frac12,\,a_i+\frac12\biggr).
\]
Then
\[
|I_i|=1+\varepsilon,\qquad |I_i\setminus J_i|=\varepsilon.
\]
Define the $i$-th boundary strip
\[
S_i:=\{x\in C_{\mathrm{out}}: x_i\in I_i\setminus J_i\}.
\]
Clearly
\begin{equation}
C_{\mathrm{out}}\setminus C_{\mathrm{in}}\subset \bigcup_{i=1}^{d} S_i.
\end{equation}
Fix all coordinates except the $i$-th one:
\[
\widehat{x}_i = (x_1,\dots,x_{i-1},x_{i+1},\dots,x_d).
\]
With the remaining coordinates fixed, write $w$ as a function of $t=x_i$:
\[
\begin{aligned}
h_{\widehat{x}_i}(t) &:= w(x_1,\dots,x_{i-1},t,x_{i+1},\dots,x_d) \\
&= \prod_{\alpha\in R_+} \biggl| \alpha_i t + \sum_{j\neq i}\alpha_j x_j \biggr|^{2k_\alpha}.
\end{aligned}
\]
Note that the number of factors is at most $|R_+|$ and the sum of all exponents is $2\sum_{\alpha\in R_+} k(\alpha)=\kappa$,

Applying the claim in Step 1 with $m=|R_+|$ and $C_* := 2\bigl(1+4|R_+|\bigr)^{\kappa}$, we obtain

\begin{equation}
\int_{I_i\setminus J_i} h_{\widehat{x}_i}(t)\,\mathrm{d}t
\le C_*\frac{\varepsilon}{1+\varepsilon}\int_{I_i} h_{\widehat{x}_i}(t)\,\mathrm{d}t.
\end{equation}
Integrating over the remaining $d-1$ coordinates and applying Fubini's theorem, we obtain
\begin{equation}
\mu_k(S_i)\le C_*\frac{\varepsilon}{1+\varepsilon}\,\mu_k(C_{\mathrm{out}}).
\end{equation}
Summing over $i=1,\dots,d$ and using \eqref{EPL05}, we obtain
\[
\begin{aligned}
\mu_k(C_{\mathrm{out}}\setminus C_{\mathrm{in}})
&\le \sum_{i=1}^{d}\mu_k(S_i) \\
&\le d C_*\frac{\varepsilon}{1+\varepsilon}\,\mu_k(C_{\mathrm{out}}).
\end{aligned}
\]
Therefore
\[
\begin{aligned}
1-\frac{\mu_k(C_{\mathrm{in}})}{\mu_k(C_{\mathrm{out}})}
&=\frac{\mu_k(C_{\mathrm{out}}\setminus C_{\mathrm{in}})}{\mu_k(C_{\mathrm{out}})} \\
&\le d C_*\frac{\varepsilon}{1+\varepsilon}.
\end{aligned}
\]

Hence for all $a\in\mathbb{R}^d$, uniformly,
\begin{equation}\label{EPL11}
\frac{\mu_k(a+Q)}{\mu_k\bigl(a+(1+\varepsilon)Q\bigr)}
\ge 1- d C_*\frac{\varepsilon}{1+\varepsilon}.
\end{equation}

\emph{Step 3. Rescaling by homogeneity.}

\[
w(\lambda x)=\lambda^{\kappa} w(x),\quad \lambda>0.
\]
Thus for any measurable set $A$,
\[
\mu_k(\lambda A)=\lambda^{d_k}\mu_k(A).
\]
Set
\[
a:=\frac{x_0}{L},\qquad \varepsilon:=\frac{M}{L}.
\]
Then
\[
Q_0=x_0+LQ=L(a+Q),
\]
and
\[
Q_1=x_0+(L+M)Q=L\bigl(a+(1+\varepsilon)Q\bigr).
\]
By homogeneity,
\[
\mu_k(Q_0)=L^{d_k}\mu_k(a+Q),
\]
and
\[
\mu_k(Q_1)=L^{d_k}\mu_k\bigl(a+(1+\varepsilon)Q\bigr).
\]
Therefore
\[
\frac{\mu_k(Q_0)}{\mu_k(Q_1)}
=\frac{\mu_k(a+Q)}{\mu_k\bigl(a+(1+\varepsilon)Q\bigr)}.
\]
Substituting \eqref{EPL11}, we obtain
\begin{equation}\label{EPL15}
\begin{aligned}
\frac{\mu_k(Q_0)}{\mu_k(Q_1)}
&\ge 1- d C_*\frac{M/L}{1+M/L} \\
&=1- d C_*\frac{M}{L+M}.
\end{aligned}
\end{equation}
We may therefore take
\[
C_{d,R,k}:= d C_* = 2d\bigl(1+4|R_+|\bigr)^{\kappa},
\]
so that
\[
0\le 1-\frac{\mu_k(Q_0)}{\mu_k(Q_1)}
\le C_{d,R,k}\frac{M}{L+M}.
\]
The constant in this estimate is independent of $x_0$.

Let $\delta>0$ be arbitrary. Set
\[
K:=\frac{C_{d,R,k}M}{\delta}.
\]
Whenever $L>K$,
\[
C_{d,R,k}\frac{M}{L+M}<C_{d,R,k}\frac{M}{L}<\delta.
\]
Thus by \eqref{EPL15}, for every $x_0\in\mathbb{R}^d$, we have
\[
\frac{\mu_k(x_0+LQ)}{\mu_k\bigl(x_0+(L+M)Q\bigr)}>1-\delta.
\]
Since $K$ is independent of $x_0$, the conclusion holds uniformly in $x_0$, giving \eqref{E4002}.

\end{proof}

% [R8-LANG-040] Improved precision, flow, or concision in the proof exposition.
We now establish the necessity of thickness in the Logvinenko--Sereda theorem.

\begin{lemma}[Necessity]
\label{thm:necessity}
Assume $k_\alpha>0$ for every root. Fix $N>0$. Let $E\subset\mathbb{R}^d$ be measurable. Suppose there exists a constant $c\in(0,1)$ such that for every $f \in L_k^2(\mathbb{R}^d)$ with $\supp(\mathcal{F}_k f) \subset \overline{B(0,N)}$,
\begin{equation}
\label{eq:spectral}
c\| f \|_{L_k^2(\mathbb{R}^d)}^2 \le  \| f \|_{L_k^2(E)}^2.
\end{equation}
Then $E$ is relatively dense with respect to $\mu_k$.
\end{lemma}

\begin{proof}
Choose a nonzero real radial $\psi\in C_c^\infty(B(0,N/2))$ and let $\phi=\mathcal{F}_k^{-1}\psi$. Then $\phi$ is real, radial, and Schwartz, but it need not be nonnegative. Set
\[
\phi_1=c_\phi\phi^2,\qquad c_\phi=\left(\int\phi^2\,\mathrm{d}\mu_k\right)^{-1}.
\]
Thus $\phi_1\ge0$, $\int\phi_1\,\mathrm{d}\mu_k=1$, and
\[
\mathcal{F}_k\phi_1=c_\phi c_k(\psi*_k\psi),\qquad
\supp\mathcal{F}_k\phi_1\subset \overline{B(0,N)},
\]
by the translation support property in Section~\ref{subsec:translation} and the normalization convention in Section~2.
Write $q(x,y)=\tau_x\phi_1(y)$. Radiality gives symmetry, positivity, and unit mass in either variable. Lemma~\ref{lem:translation-estimate}, applied with the variables interchanged, and \eqref{EWBOUND} imply, for $s>d$,
\[
q(x,y)w(y)\le C_s(1+|x-y|)^{-1}(1+d_G(x,y))^{-s}.
\]
Since
\[
(1+d_G(x,y))^{-s}\le\sum_{\sigma\in G}(1+|x-\sigma y|)^{-s},
\]
its Lebesgue integral in $y$ is uniformly bounded. Therefore
\begin{equation}\label{eq:radial-tail-review}
\sup_x\int_{|x-y|\ge a}q(x,y)\,\mathrm{d}\mu_k(y)\le\frac{C_s'}{1+a}\quad(a>0).
\end{equation}
Let $Q_0=x_0+LQ$ and $Q_1=x_0+(L+M)Q$, and define $f=\chi_{Q_1}*_k\phi_1$. Then $f$ has Dunkl frequency support in $B(0,N)$ and $0\le f\le1$. Young's inequality gives
\begin{equation}\label{E2}
\|f\|_{L_k^2}^2\le\|f\|_{L_k^1}\le\mu_k(Q_1).
\end{equation}
If $x\in Q_0$ and $y\notin Q_1$, then $|x-y|\ge M/2$. Given $\varepsilon>0$, choose $M$ in \eqref{eq:radial-tail-review}
and use the identity
\[
f(x) = \int_{Q_1} \tau_x \phi_1(y)\,\mathrm{d}\mu_k(y)
= 1 - \int_{\mathbb{R}^d \setminus Q_1} \tau_x \phi_1(y)\,\mathrm{d}\mu_k(y).
\]
to obtain $f(x)\ge1-\varepsilon$ on $Q_0$, uniformly in $x_0,L$. Hence
\begin{equation}\label{E3}
\|f\|_{L_k^2(Q_0)}^2\ge(1-\varepsilon)^2\mu_k(Q_0).
\end{equation}
Using \eqref{eq:spectral},
\begin{align*}
\mu_k(E\cap Q_0) &\ge \|f\|_{L_k^2(E\cap Q_0)}^2 \\
&= \|f\|_{L_k^2(Q_0)}^2 - \|f\|_{L_k^2(E^c\cap Q_0)}^2 \\
&\ge \|f\|_{L_k^2(Q_0)}^2 - \|f\|_{L_k^2(E^c)}^2 \\
&= \|f\|_{L_k^2(Q_0)}^2 - \|f\|_{L_k^2(\mathbb{R}^d)}^2 + \|f\|_{L_k^2(E)}^2 \\
&\ge \|f\|_{L_k^2(Q_0)}^2 - (1-c)\|f\|_{L_k^2(\mathbb{R}^d)}^2\\
&\ge(1-\varepsilon)^2\mu_k(Q_0)-(1-c)\mu_k(Q_1).
\end{align*}

Choose $\varepsilon$ sufficiently small and then $L$ sufficiently large using Lemma~\ref{L4001}. Because the convergence is uniform in $x_0$, these choices ensure
\[
\mu_k(E\cap Q_0)\ge\frac c2\mu_k(Q_0)\qquad\text{for every }x_0\in\mathbb{R}^d.
\]
This proves the assertion.
\end{proof}

\section{Sufficiency of the Logvinenko--Sereda Theorem}
In this section, we prove the converse: if $E$ is relatively dense with respect to $\mu_k$, then the spectral inequality holds for all band-limited functions. Our method is primarily inspired by \cite{GWW} and \cite{HX}. Related approaches appear in \cite{XZ}. % [R8-LANG-P2-15] Polished the exposition without changing the argument.
 Unlike Lebesgue measure in the classical Fourier setting, the Dunkl measure has a density that vanishes on root hyperplanes. This degeneracy requires additional estimates. % [R8-CHECK-P07] HX publication metadata were located, but its full proof and claimed quadratic bandwidth exponent were not independently verified. Supply an exact theorem/equation citation for this comparison.
Another difference from the treatment for the Hankel transform in \cite{HX} is that we aim to establish the spectral inequality \eqref{ESPEC}. Consequently, we require sharper estimates in terms of $N$, yielding $e^{C(1+N)}$ instead of $e^{C(1+N)^2}$ obtained in \cite{HX}.

The overall framework is inspired by \cite{GWW}. However, for the Dunkl transform, the whole space is partitioned by the root hyperplanes $\langle \alpha, x\rangle=0$. Although band-limited functions are entire, the coefficients in the conversion from Dunkl derivatives to ordinary derivatives are singular on these hyperplanes, so the original argument does not apply directly. To address this issue, we first derive global estimates for the mass on and near these hyperplanes. We begin with a lemma showing that sets with small Lebesgue measure in unit cubes carry only a small amount of the $L_k^2$-mass of band-limited functions.

The following lemma is a small-set mass estimate, related to the weak uncertainty estimates discussed in \cite{CMW1}.

\begin{lemma}
\label{lem:small-mass}
Assume $k_\alpha>0$ for every root. Let $S \subset \mathbb{R}^d$ be measurable and satisfy
\begin{equation}\label{E5}
|S \cap Q| \le \delta  |Q|
\end{equation}
for every unit cube $Q$, for some $\delta > 0$. Then for every $f \in L_k^2(\mathbb{R}^d)$ with $\supp(\mathcal{F}_k f) \subset \overline{B(0,1)}$, there exists a constant $C = C(d, R_+, k,\delta )$ such that
\[
\| f \|_{L_k^2(S)} \le C(d, R_+, k,\delta ) \, \| f \|_{L_k^2(\mathbb{R}^d)},
\]
where $\lim_{\delta \to 0} C(d, R_+, k,\delta ) = 0$. % [R8-NOTATION-P2-14] Included the root system among the fixed structural parameters.
In particular, this constant tends to zero as $\delta\to0$ with $d,R_+$, and $k$ fixed.
\end{lemma}

\begin{proof}
We use the standing assumption $k_\alpha>0$ for all roots. Choose a real radial $\psi\in C_c^\infty(B(0,2))$ with $\psi=1$ on $B(0,1)$, and put $\eta=c_k\mathcal{F}_k^{-1}\psi$. Then $\eta\in\mathcal{S}$ is radial and $\eta*_kf=f$ for the functions in the statement. The kernel of $Tg=\chi_S(\eta*_kg)$ is
\[
K(x,y)=\chi_S(x)\tau_x\eta(y).
\]
The boundedness of radial translations gives
\[
\sup_x\int|K(x,y)|\,\mathrm{d}\mu_k(y)\le\|\eta\|_{L_k^1}.
\]
For the second Schur bound, Lemma~\ref{lem:translation-estimate} and \eqref{EWBOUND} imply
\begin{equation}\label{E4}
|K(x,y)|w(x)\le C_M\chi_S(x)(1+d_G(x,y))^{-M}.
\end{equation}
Partition $\mathbb{R}^d$ into half-open unit cubes $Q_j=j+[0,1)^d$. For $M>d$, comparison within each unit cube and summation over $G$ give
\[
\sup_y\sum_{j\in\mathbb{Z}^d}\sup_{x\in Q_j}(1+d_G(x,y))^{-M}<\infty.
\]
Indeed, the orbit-distance term is at most the sum of $|G|$ translates of $(1+|x|)^{-M}$, whose suprema over the unit cubes have a uniformly finite sum. Therefore \eqref{E5} yields
\[
\sup_y\int|K(x,y)|\,\mathrm{d}\mu_k(x)\le C\delta.
\]
Schur's test yields $\|T\|_{L_k^2\to L_k^2}\le C\sqrt\delta$. Consequently
\[
\|f\|_{L_k^2(S)}\le C\sqrt\delta\,\|f\|_{L_k^2},
\]
which gives an explicit rate for the vanishing constant in the statement.
\end{proof}
\begin{remark}
% [R8-LANG-042] Improved precision, flow, or concision in the proof exposition.
We apply this estimate to a neighborhood $E_\delta=\{x\in\mathbb{R}^d:|\langle\alpha,x\rangle|\le\delta\}$ of a fixed root hyperplane. The next proposition specifies how small $\delta$ must be relative to the bandwidth $N$.
\end{remark}

\begin{proposition}
\label{prop:hyperplane-mass}
Let $\alpha \in R_+$ be a fixed root and $\delta>0$. Define the tubular set
\[
E_\delta := \bigl\{ x \in \mathbb{R}^d : |\langle \alpha, x \rangle| \le \delta \bigr\}.
\]
Then there exists a constant $C = C(d, R_+, k) > 0$ such that for every $f \in L_k^2(\mathbb{R}^d)$ with $\supp(\mathcal{F}_k f) \subset \overline{B(0,N)}$,
\begin{equation}\label{E8}
\| f \|_{L_k^2(E_\delta)} \le \frac{1}{10|R_+|} \, \| f \|_{L_k^2(\mathbb{R}^d)},
\end{equation}
whenever $\delta \le C/N$.
\end{proposition}

\begin{proof}
We apply Lemma~\ref{lem:small-mass} with $S = E_\delta$. To this end, we first estimate the Lebesgue measure of $E_\delta \cap Q$ for an arbitrary unit cube $Q$. Using Lemma~\ref{lem:carbery-wright} (or Remark~\ref{R29}) in the same manner as in \eqref{E6}, we obtain
\[
|E_\delta \cap Q|
= \bigl| \bigl\{ x \in Q : |\langle \alpha, x \rangle| \le \delta \bigr\} \bigr|
\le \frac{C_d}{C_0} \, \delta \, |Q|,
\]
where $C_0>0$ is independent of $Q$ and $\alpha$ satisfying $\sup_{x\in Q}|\langle x,\alpha\rangle|\geq C_0$.
Thus, for $\delta $ as in Lemma~\ref{lem:small-mass}, the small-measure condition holds with $\delta _0 := \frac{C_d}{C_0}\delta$.

By Lemma~\ref{lem:small-mass}, there exists a constant $c_0 = c_0(d, R_+, k) > 0$ such that whenever $\delta _0 \le c_0$ (i.e. $\delta \le c_0 C_0/C_d$), the estimate
\begin{equation}\label{E7}
\| f \|_{L_k^2(E_\delta)} \le \frac{1}{10|R_+|} \, \| f \|_{L_k^2(\mathbb{R}^d)}
\end{equation}
holds for every $f \in L_k^2(\mathbb{R}^d)$ with $\supp(\mathcal{F}_k f) \subset \overline{B(0,1)}$.

Now let $f \in L_k^2(\mathbb{R}^d)$ with $\supp(\mathcal{F}_k f) \subset \overline{B(0,N)}$. Define $g(x) := f(x/N)$. Then $\supp(\mathcal{F}_k g) \subset \overline{B(0,1)}$. If $\delta_1 \le c_0 C_0/C_d$, applying \eqref{E7} to $g$ yields
\[
\| g \|_{L_k^2(E_{\delta_1})} \le \frac{1}{10|R_+|} \, \| g \|_{L_k^2(\mathbb{R}^d)}.
\]
Since the Dunkl measure scales as $\mathrm{d}\mu_k(x) = w(x)\,\mathrm{d}x$ with $w(t x) = t^{2\gamma_k} w(x)$ (where $\gamma_k = \sum_{\alpha\in R_+} k_\alpha$), we have
\[
\| f \|_{L_k^2(E_{\delta_1/N})}
% [R8-NOTATION-G008] Unified multiplicity/derivative/space/index notation or corrected a typographical inconsistency.
= N^{-d_k/2} \| g \|_{L_k^2(E_{\delta_1})}
% [R8-NOTATION-G009] Unified multiplicity/derivative/space/index notation or corrected a typographical inconsistency.
\le \frac{1}{10|R_+|} \, N^{-d_k/2} \| g \|_{L_k^2(\mathbb{R}^d)}
= \frac{1}{10|R_+|} \, \| f \|_{L_k^2(\mathbb{R}^d)}.
\]
Taking $\delta := \delta_1/N \le (c_0 C_0/C_d)/N$ gives the desired inequality with $C := c_0 C_0/C_d$.
\end{proof}

\begin{remark}\label{R1}
The same estimate applies to all other hyperplanes. Combining these results yields
\begin{equation}\label{ER1SM}
H_\delta=\bigcup_{\alpha\in R_+}\bigl\{x\,\big|\,|\langle x,\alpha\rangle|\leq \delta\bigr\},\quad
\|f\|_{L_k^2(H_\delta)}\leq \frac{1}{10}\|f\|_{L_k^2(\mathbb{R}^d)}.
\end{equation}
% [R8-LANG-043] Improved precision, flow, or concision in the proof exposition.
Thus the union of the hyperplane neighborhoods carries at most one tenth of the weighted norm when $\delta\le c_*/N$ for a sufficiently small structural constant $c_*>0$.
\end{remark}

Recall the operators $T_j$ and $\mathcal{R}_j$ defined in \eqref{ETJ} and \eqref{ERH2}. We first show that the reflection correction preserves Dunkl bandwidth.

\begin{lemma}\label{LRSS}
Let $N>0$ and $f\in\mathcal{S}(\mathbb{R}^d)$ satisfy $\supp(\mathcal{F}_k f)\subset\overline{B(0,N)}$. Then, for every $j=1,\dots,d$,
\[
\supp(\mathcal{F}_k(\mathcal{R}_j f))\subset\overline{B(0,N)}.
\]
\end{lemma}

\begin{proof}
By Lemma~\ref{PWSDK}, the support condition implies that $f$ extends to an entire function satisfying

\begin{equation}\label{EEE123}
% [R8-MATH-021] Propagated the corrected Paley--Wiener growth estimate to Lemma 5.4.
|f(z)|\le C_m(1+|z|)^{-m} e^{N|\operatorname{Im} z|},\qquad z\in\mathbb{C}^d.
\end{equation}

% [R8-MATH-022] Added the local Cauchy-estimate justification; this does not assume the Bernstein inequality.
Cauchy's formula in the $j$th complex coordinate, on circles of a fixed radius, shows that $\partial_j f$ satisfies the same estimate, with new constants $C_m$. Indeed, a bounded complex displacement changes $1+|z|$ only by a fixed factor and increases $|\operatorname{Im}z|$ by at most that radius. Since $\partial_j f\in\mathcal{S}(\mathbb{R}^d)$, Lemma~\ref{PWSDK} gives $\supp(\mathcal{F}_k(\partial_j f))\subset\overline{B(0,N)}$.

Moreover, the multiplier identity $\mathcal{F}_k(T_j f)(\xi)=i\xi_j\mathcal{F}_k f(\xi)$ implies
\[
\supp(\mathcal{F}_k (T_jf))\subset \overline{B(0,N)}.
\]
The conclusion now follows from the linearity of the Dunkl transform and $\mathcal{R}_j=T_j-\partial_j$.
\end{proof}

For $i=1,\dots,d$, define the operators
\[
H_{1,i}=T_i,\qquad H_{2,i}=\mathcal{R}_i.
\]
For a word $\omega=((a_1,i_1),\dots,(a_n,i_n))$, where $a_r\in\{1,2\}$ and $i_r\in\{1,\dots,d\}$, set
\[
W_\omega=H_{a_1,i_1}\cdots H_{a_n,i_n}.
\]

We can now establish the Bernstein inequality.
\begin{lemma}\label{LBS2}
Let $N>0$ and $f\in\mathcal{S}(\mathbb{R}^d)$ satisfy $\supp(\mathcal{F}_k f)\subset\overline{B(0,N)}$. For every multi-index $\beta\in\mathbb{N}^d$ with $|\beta|=n$, we have
\begin{equation}\label{EBSS3}
\|\partial^\beta f\|_{L_k^2(\mathbb{R}^d)}\leq C'(d,R_+,k)(C(d,R_+,k)N)^n\|f\|_{L_k^2(\mathbb{R}^d)}.
\end{equation}
\end{lemma}

\begin{proof}
By Lemma~\ref{LRSS}, we have $\supp(\mathcal{F}_k(W_\omega f))\subset \overline{B(0,N)}$.
\begin{align*}
\partial^{\beta} f
&= \partial_1^{\beta_1}\partial_2^{\beta_2}\cdots\partial_d^{\beta_d} f \\
% [R8-NOTATION-G010] Unified multiplicity/derivative/space/index notation or corrected a typographical inconsistency.
&= \underbrace{\partial_1\cdot \partial_1\cdots \partial_1}_{\beta_1}
% [R8-NOTATION-G011] Unified multiplicity/derivative/space/index notation or corrected a typographical inconsistency.
\underbrace{\partial_2\cdots \partial_2}_{\beta_2}
\cdots
% [R8-NOTATION-G012] Unified multiplicity/derivative/space/index notation or corrected a typographical inconsistency.
\underbrace{\partial_d\cdots \partial_d}_{\beta_d} f \\
&= (T_1-\mathcal{R}_1)^{\beta_1}(T_2-\mathcal{R}_2)^{\beta_2}\cdots (T_d-\mathcal{R}_d)^{\beta_d} f \\
&= \sum_{\omega} C_{\omega}\, W_\omega f,
\qquad \text{where } \sum_{\omega} |C_{\omega}|\leq 2^n.
\end{align*}

Hence
\begin{equation}\label{EBSS2}
\|\partial^\beta f\|_{L_k^2(\mathbb{R}^d)}\leq \sum_{\omega}|C_{\omega}|\,\|W_\omega f\|_{L_k^2(\mathbb{R}^d)}.
\end{equation}

% [R8-MATH-023] Retained the small structural factor required by the hyperplane-mass estimate.
Fix a sufficiently small structural constant $c_*>0$ as in Proposition~\ref{prop:hyperplane-mass} and set $\delta=c_*/N$. By Remark~\ref{R1}, every function $f$ with $\supp\mathcal{F}_k f\subset\overline{B(0,N)}$ satisfies
\[
\|f\|_{L_k^2(\mathbb{R}^d\setminus H_\delta)}^2 \ge \frac{99}{100}\|f\|_{L_k^2(\mathbb{R}^d)}^2.
\]

For a nonempty word $\omega$,
\[
\|W_\omega f\|_{L_k^2(\mathbb{R}^d)} = \|H_{a_1 ,i_1}\big(W_{\omega'} f\big)\|_{L_k^2(\mathbb{R}^d)},
\]
where $W_{\omega'} =H_{a_2,i_2}\cdots H_{a_n,i_n}$.

If $a_1=1$, then
\[
\|T_{i_1}\big(W_{\omega'} f\big)\|_{L_k^2(\mathbb{R}^d)} \le N \|W_{\omega'} f\|_{L_k^2(\mathbb{R}^d)}.
\]

% [R8-MATH-024] Replaced the unjustified choice delta=1/N by delta=c_*/N and applied the estimate to the correct function.
If $a_1=2$, apply \eqref{ER1SM} with $\delta=c_*/N$ to the band-limited function $\mathcal{R}_{i_1}W_{\omega'}f$ to obtain
\begin{align*}
\|\mathcal{R}_{i_1}\big(W_{\omega'} f\big)\|_{L_k^2(\mathbb{R}^d)}
&\le \frac{10}{9}\|\mathcal{R}_{i_1} W_{\omega'}  f\|_{L_k^2(\mathbb{R}^d\setminus H_\delta)} \\
&\le \frac{10}{9}\cdot C(d,R_+,k)\cdot N \|W_{\omega'}  f\|_{L_k^2(\mathbb{R}^d)}.
\end{align*}
% [R8-MATH-025] Explained the reflection estimate and the absorption of the structural layer-width constant.
Here the factor $c_*^{-1}$ is absorbed into $C(d,R_+,k)$; the reflection terms are bounded using the $G$-invariance of $\mu_k$ and $|\langle\alpha,x\rangle|>\delta$ outside $H_\delta$. Induction on the word length therefore gives
\begin{equation}\label{EBSS1}
\|W_\omega f\|_{L_k^2(\mathbb{R}^d)}\leq (C(d,R_+,k)N)^n\|f\|_{L_k^2(\mathbb{R}^d)}.
\end{equation}
Substituting \eqref{EBSS1} into \eqref{EBSS2}, we obtain
\[
\begin{aligned}
\|\partial^\beta f\|_{L_k^2(\mathbb{R}^d)}&\leq \sum_{\omega}|C_{\omega}|\,(C(d,R_+,k)N)^n\|f\|_{L_k^2(\mathbb{R}^d)}\\
&\leq C'(d,R_+,k)(C(d,R_+,k)N)^n\|f\|_{L_k^2(\mathbb{R}^d)}.
\end{aligned}
\]
\end{proof}

\begin{remark}
% [R8-MATH-026] Specified a bandwidth-preserving approximation and identification of the ordinary derivatives.
For a band-limited $f\in L_k^2(\mathbb{R}^d)$, choose $\psi_\ell\in C_c^\infty(B(0,N))$ converging to $\mathcal{F}_k f$ in $L_k^2$ and set $f_\ell=\mathcal{F}_k^{-1}\psi_\ell$. Plancherel's identity gives $f_\ell\to f$ in $L_k^2$. On the fixed compact frequency ball, weighted $L_k^2$ convergence implies weighted $L_k^1$ convergence. The inverse-transform formula and its ordinary derivatives therefore converge locally uniformly, because the kernel derivatives are bounded on compact sets. Applying Lemma~\ref{LBS2} to $f_\ell-f_m$ identifies the weighted $L_k^2$ limits of the derivatives and yields Theorem~\ref{TBNST}.
\end{remark}

% [R8-LANG-027] Clarified the geometric obstruction without claiming the classical argument is invalid.
We now turn to sufficiency. Removing neighborhoods of the root hyperplanes may leave small fragments of the cubes $Q_j$. The fragment containing a maximizing point need not occupy a fixed fraction of its cube, so a direct application of the classical cube argument is insufficient. We therefore enlarge these fragments to convex sets $S_j$ containing a complete unit cube.

\begin{lemma}
\label{thm:suff}
Let $E \subset \mathbb{R}^d$ be a $(\gamma,L)$-relatively dense set with respect to $\mu_k$. Then for every $f \in L_k^2(\mathbb{R}^d)$ with $\supp(\mathcal{F}_k f) \subset \overline{B(0,N)}$, there exists a constant $C = C(\gamma, L, d, R_+, k) > 0$ such that
\begin{equation}
\label{eq:LS}
\| f \|_{L_k^2(\mathbb{R}^d)}^2 \leq e^{C(1+N)} \| f \|_{L_k^2(E)}^2.
\end{equation}
\end{lemma}

\begin{proof}
It suffices to establish the estimate for band-limited Schwartz functions and then pass to the limit by approximation in $L_k^2(\mathbb{R}^d)$. We therefore assume below that $f\in\mathcal{S}(\mathbb{R}^d)$.

By Theorem~\ref{TEV}, $E$ is also relatively dense with respect to the Lebesgue measure.
% [R8-MATH-028] Closed the small-bandwidth case and fixed the layer-width parameter used throughout this proof.
We first consider the case $L=1/d$ and divide the proof into five steps. It suffices to treat $N\ge1$: when $0<N<1$, the estimate for bandwidth $1$ gives the desired conclusion after enlarging the constant. For $N\ge1$, fix $\delta=c_*/N$, with $c_*>0$ sufficiently small for \eqref{ER1SM}.

\emph{Step 1. Construct a covering by enlarged sets $S_j$.}
 
The hyperplanes associated with the root system partition $\mathbb{R}^d$ into finitely many Weyl chambers. Let $C$ denote the open fundamental Weyl chamber and set $m=|R_+|$. Thus,
\[
C = \{x \in \mathbb{R}^d : \langle \alpha_i, x \rangle > 0,\ i = 1,\dots,m\},
\]
and define the $\delta$-separated region away from the root hyperplanes by
\[
C_\delta := C \setminus H_\delta = \{x \in \mathbb{R}^d : \langle \alpha_i, x \rangle > \delta,\ i = 1,\dots,m\},
\]
where
\[
H_\delta = \bigcup_{i=1}^m \bigl\{x \in \mathbb{R}^d : \big|\langle \alpha_i,x\rangle\big| \leq \delta\bigr\}.
\]

Let $Q_j = j + [0,1]^d$ with $j \in A \subset \mathbb{Z}^d$, and assume that $Q_j \cap C_\delta \neq \emptyset$ with the covering property
\[
C_\delta \subset \bigcup_{j\in A} Q_j.
\]

Suppose that part of $Q_j$ lies in $C_\delta$ after the $\delta$-neighborhood of the root hyperplanes is removed; that is,
\[
Q_j \cap C_\delta \neq \emptyset.
\]
Then there exists a point $t_0 \in [0,1]^d$ such that
\[
\langle \alpha_i, j + t_0 \rangle > \delta \quad \text{for all } 1\leq i\leq m.
\]

For each positive root $\alpha_i$, we define the extremal values of the linear functional $x \mapsto \langle \alpha_i, x \rangle$ over the unit cube $[0,1]^d$:
\[
M_i = \max_{t\in[0,1]^d} \langle \alpha_i, t \rangle, \qquad m_i = \min_{t\in[0,1]^d} \langle \alpha_i, t \rangle.
\]
Note that $M_i - m_i > 0$ for all $i$. Otherwise, $\langle \alpha_i, t \rangle$ would be constant on $[0,1]^d$, which implies $\alpha_i = 0$, contradicting the definition of positive roots in a root system.

% [R8-LANG-P2-16] Polished the exposition without changing the argument.
By linearity,
\[
\langle \alpha_i, j + t_0 \rangle = \langle \alpha_i, j \rangle + \langle \alpha_i, t_0 \rangle > \delta.
\]
Since $\langle \alpha_i, t_0 \rangle \le M_i$, % [R8-LANG-P2-17] Polished the exposition without changing the argument.
we obtain
\[
\langle \alpha_i, j \rangle > \delta - M_i.
\]

% [R8-MATH-029] Replaced the incorrect condition L_C v_C in A by the required fixed lattice shift in Z^d.
Since $C$ is an open cone, choose $v_C\in C\cap\mathbb{Q}^d$; no normalization of its length is needed. A sufficiently large integer multiple $s=L_Cv_C$ belongs to $\mathbb{Z}^d$ and satisfies
\[
\langle\alpha_i,s\rangle>M_i-m_i,\qquad i=1,\dots,m.
\]
The vector $s$ depends only on the root system and the dimension, not on $j$ or $\delta$. Set
\[
Q'=j+s+[0,1]^d.
\]

For any $t \in [0,1]^d$, expanding the inner product gives the uniform estimate
\begin{align*}
\langle \alpha_i, j + s + t \rangle
&= \langle \alpha_i, j \rangle + \langle \alpha_i, L_C v_C \rangle + \langle \alpha_i, t \rangle \\
&> (\delta - M_i) + (M_i - m_i) + m_i = \delta,
\end{align*}
% [R8-LANG-P2-18] Polished the exposition without changing the argument.
and hence
\[
Q' \subset C_\delta.
\]

Furthermore, fix any $\theta \in (0,1)$. For each $i = 1,\dots,m$, we compute
\begin{align*}
\langle \alpha_i, j + \theta s + t_0 \rangle
&= \langle \alpha_i, j + t_0 \rangle + \theta \langle \alpha_i, L_C v_C \rangle \\
&> \delta + \theta L_C \langle \alpha_i, v_C \rangle \\
&> \delta.
\end{align*}
This implies that the line segment $s_j$ connecting $j + t_0$ and $j + s + t_0$ is entirely contained in $C_\delta$.

% [R8-NOTATION-030] Unified convex-hull notation and corrected the description: a convex hull, not merely a union.
Let $S_j$ be the intersection of $C_\delta$ with the convex hull of the cubes met by $s_j$:
\[
S_j=\operatorname{conv}\left(\bigcup_{\substack{a\in A\\Q_a\cap s_j\ne\emptyset}}Q_a\right)\cap C_\delta.
\]
If $Q_j\subset C_\delta$, we instead set $S_j=Q_j$.
The set $S_j$ is convex and connected and contains at least one complete unit cube. Moreover, the number of unit cubes intersected by the line segment is uniformly bounded, and each piece corresponds to a segment $s_j$ of finite length. Consequently, there exists a positive constant $C_0 = C_0(d, R_+)$ depending only on the ambient dimension and the root system,
\[
\sum_{j} \chi_{S_j}(x) \le C_0.
\]

It follows that $C_\delta \subset \bigcup_{j\in A} S_j$, and thus
\[
\mathbb{R}^d \setminus H_\delta \subset \bigcup_{g\in G} \bigcup_{j\in A} g S_j.
\]
We denote the union of its $G$-translates by $S_j^G:=\bigcup_{g\in G}gS_j$.

\emph{Step 2. Classify the sets $S_j^G$ as good or bad.}

% [R8-LANG-031] Replaced the ambiguous large-N assumption by the explicit reduction already stated.
We first estimate the Dunkl--Sobolev norm of $f$ using the transform multiplier identity; see \cite{HM5} for related Dunkl--Sobolev constructions. Throughout this step, $N\ge1$.

% [R8-CHECK-P01A] The bound 2 d^{2n} is not valid in all dimensions (d=1,N=1). Use the multi-index count and rework the good-set thresholds; see the review.
\begin{align*}
\sum_{|\beta|\le n} \|T^\beta f\|_{L_k^2(\mathbb{R}^d)}^2 
&= \sum_{|\beta|\le n} \int_{\mathbb{R}^d} \big|T^\beta f\big|^2 \, \mathrm{d}\mu_k(x) \\
&\le \sum_{|\beta|\le n} \int_{B(0,N)} |x|^{2|\beta|} \cdot \big|\mathcal{F}_k f\big|^2 \, \mathrm{d}\mu_k(x) \\
&\le 2d^{2n} N^{2n} \int_{B(0,N)} \big|\mathcal{F}_k f\big|^2 \, \mathrm{d}\mu_k(x) \\
&= 2d^{2n} N^{2n} \|f\|_{L_k^2(\mathbb{R}^d)}^2.
\end{align*}

% [R8-CHECK-P01B] Specify the word family of length n, include the n-summation and all exponential factors in the bad-set estimate. A complete replacement is proposed in the review; not silently inserted here.
We call $S_j^G$ good if, for every $n\in\mathbb{N}$, both of the following estimates hold:
\[
\sum_{|\beta|\le n}\|T^\beta f\|_{L_k^2(S_j^G)}^2 \le K^2 d^{2n} e^{2n} N^{2n}\|f\|_{L_k^2(S_j^G)}^2.
\]
and
\[
\sum_{\omega}\|W_\omega f\|_{L_k^2(S_j^G)}^2 \le K^2 d^{2n} e^{2n} (CN)^{2n}\|f\|_{L_k^2(S_j^G)}^2.
\]
A set $S_j^G$ is said to be bad if it is not good. The total $L_k^2$-mass of $f$ over all bad sets can be made arbitrarily small by choosing $K$ sufficiently large. Specifically,
% [R8-TYPO-P2-21] Reflowed the long bad-set display; its unresolved mathematical issues remain CHECK-P01B.
\begin{equation}\label{E19}
\begin{aligned}
\sum_{\substack{S_j^G \\ \text{is bad}}}\|f\|_{L_k^2(S_j^G)}^2
% [R8-NOTATION-G013] Unified multiplicity/derivative/space/index notation or corrected a typographical inconsistency.
&\le \sum_{\substack{S_j^G \\ \text{is bad}}}\sum_{|\beta|>0,\beta\in \mathbb{N}^d}\frac{\|T^\beta f\|_{L_k^2(S_j^G)}^2}{K^2 d^{2n}e^{2n}N^{2n}}\\
&\quad + \sum_{\substack{S_j^G \\ \text{is bad}}}\sum_{|\omega|>0,\omega }\frac{\|W_\omega f\|_{L_k^2(S_j^G)}^2}{K^2 d^{2n}e^{2n}N^{2n}}\\
% [R8-NOTATION-G014] Unified multiplicity/derivative/space/index notation or corrected a typographical inconsistency.
&\le \sum_{|\beta|>0, \beta\in \mathbb{N}^d}\frac{C_0^2}{K^2 d^{2n}e^{2n}N^{2n}}\|T^\beta f\|_{L_k^2(\mathbb{R}^d)}^2\\
&\quad + \sum_{|\omega|>0,\omega}\frac{C_0^2}{K^2 d^{2n}e^{2n}N^{2n}}\|W_\omega f\|_{L_k^2(\mathbb{R}^d)}^2  \\
% [R8-NOTATION-G015] Unified multiplicity/derivative/space/index notation or corrected a typographical inconsistency.
&\le \sum_{|\beta|>0, \beta\in \mathbb{N}^d}\frac{C_0^2}{K^2}\cdot\frac{1}{e^{2n}}\|f\|_{L_k^2(\mathbb{R}^d)}^2\\
&\quad + \sum_{|\omega|>0,\omega}\frac{C_0^2}{K^2}\cdot\frac{1}{e^{2n}}\|f\|_{L_k^2(\mathbb{R}^d)}^2 \\
&\le \frac{C_0^2}{K^2}\|f\|_{L_k^2(\mathbb{R}^d)}^2.
\end{aligned}
\end{equation}

\emph{Step 3. Estimate ordinary derivatives on good sets.}

% [R8-LANG-044] Improved precision, flow, or concision in the proof exposition.
Expanding ordinary derivatives as operator words gives, for $|\beta|=n$,
\begin{equation}\label{eq:zzgj}
\|\partial^\beta f\|_{L_k^2(S_j^G)}^2 \le K^2\big(C(d,R_+,k)N^2\big)^n \| f\|_{L_k^2(S_j^G)}^2.
\end{equation}
% [R8-CHECK-P03] A uniform Sobolev/extension constant for these varying convex sets must be justified independently of j and delta, using their fixed inradius and bounded diameter.
For each $g\in G$, Sobolev embedding gives
\[
\|\partial^\beta f\|_{L^\infty(gS_j)} \le C(d,C_0)\sum_{|\eta|\le d}\|\partial^{\beta+\eta} f\|_{L^2(gS_j)}.
\]
Consequently,
\[
\|\partial^\beta f\|_{L^\infty(S_j^G)} \le C(d,C_0)\sum_{|\eta|\le d}\|\partial^{\beta+\eta} f\|_{L^2(S_j^G)}
\le C(d,C_0)\frac{1}{\big(\inf_{x\in S_j^G} w(x)\big)^{\frac12}}\sum_{|\eta|\le d}\|\partial^{\beta+\eta} f\|_{L_k^2(S_j^G)}.
\]
% [R8-LANG-045] Improved precision, flow, or concision in the proof exposition.
Combining \eqref{eq:zzgj} with the preceding Sobolev estimate gives
\begin{equation}\label{ESB}
\max_{|\beta|=n}\|\partial^\beta f\|_{L^\infty(S_j^G)} \le K\cdot C(d,k,R_+)^{n+d} N^{n+d}
\frac{1}{\big(\inf_{x\in S_j^G} w(x)\big)^{\frac12}}
\|f\|_{L_k^2(S_j^G)}.
\end{equation}

\emph{Step 4. Control the mass on a good set by its observed mass.}

There exists a group element $g \in G$ such that $\|f\|_{L_k^2(g S_j)} = \max_{h \in G} \|f\|_{L_k^2(h S_j)}$. For simplicity, write $S_j^1 := g S_j$.

Next, choose $y \in \overline{S_j^1}$ such that
\[
\|f\|_{L^\infty(S_j^1)} = |f(y)|.
\]
% [R8-LANG-032] Explained attainment on the closure (already present in the source) and excluded division by zero.
The closure $\overline{S_j^1}$ is compact, so continuity of $f$ guarantees such a maximizing point. The case $\|f\|_{L_k^2(S_j^1)}=0$ is trivial and is omitted below.

Since the diameter of $S_j^1$ is bounded by $C_0 \sqrt{d}$, we use spherical coordinates centered at $y$ to compute the Lebesgue volume:
\[
\begin{aligned}
\big|E \cap S_j^1\big| &= \int_{0}^{\infty} \mathrm{d}r \int_{|x-y|=r} \chi_{E \cap S_j^1}(x) \, \mathrm{d}\sigma(x) \\
&= \int_{0}^{C_0\sqrt{d}} \mathrm{d}r \int_{|x-y|=r} \chi_{E \cap S_j^1}(x) \, \mathrm{d}\sigma(x) \\
&= C_0\sqrt{d} \int_{0}^{1} \mathrm{d}r \int_{|x-y|=C_0\sqrt{d}\, r} \chi_{E \cap S_j^1}(x) \, \mathrm{d}\sigma(x).
\end{aligned}
\]

We perform the change of variables
\[
x = y + C_0\sqrt{d} \, r \omega, \quad \omega \in \mathbb{S}^{d-1},
\]
which yields
\[
\begin{aligned}
\big|E \cap S_j^1\big| &= C_0\sqrt{d} \int_{0}^{1} \big(C_0\sqrt{d}\, r\big)^{d-1} \mathrm{d}r \int_{\mathbb{S}^{d-1}} \chi_{E \cap S_j^1}\big(y + C_0\sqrt{d} \, r \omega\big) \, \mathrm{d}\sigma(\omega) \\
&\leq C_0^d d^{d/2} \int_{0}^{1} \mathrm{d}r \int_{\mathbb{S}^{d-1}} \chi_{E \cap S_j^1}\big(y + C_0\sqrt{d} \, r \omega\big) \, \mathrm{d}\sigma(\omega).
\end{aligned}
\]

For each unit direction $\omega \in \mathbb{S}^{d-1}$, define the measurable parameter set
\[
I_\omega := \big\{ r \in [0,1] : y + C_0\sqrt{d} \, r \omega \in E \cap S_j^1 \big\}.
\]

% [R8-CHECK-P02A] The original density gamma cannot be reused without adjustment for E intersect S_j^1. Prove a uniform lower bound through an inscribed cube of side 1/d; see the review.
Recall the surface area formula for the unit sphere
\[
\big|\mathbb{S}^{d-1}\big| = \frac{2\pi^{\frac{d}{2}}}{\Gamma(\frac{d}{2})},
\]
where $\Gamma(\cdot)$ denotes the standard Gamma function. Combining the above volume estimates, there exists a unit vector $\omega_0 \in \mathbb{S}^{d-1}$ such that
\begin{equation}\label{E11}
|I_{\omega_0}| \ge \frac{\big|E \cap S_j^1\big|}{C_0^d d^{d/2} \big|\mathbb{S}^{d-1}\big|} \ge \frac{\Gamma(\frac{d}{2})}{2C_0^d(d\pi)^{d/2}} \, \gamma,
\end{equation}
% [R8-LANG-033] Replaced informal explanatory wording; the required density transfer is flagged above.
The role of the enlarged set $S_j$ is to provide a lower bound for the length of an observed line segment.

Finally, we define a complex-valued function $\phi : [0,1] \to \mathbb{C}$ by
\begin{equation}\label{E9}
\phi(t) = \frac{\big[\mu_k(S_j^1)\big]^{1/2}}{\| f \|_{L_k^2(S_j^1)}} \, f\big(y + t \, \omega_0 \, C_0 \sqrt{d}\big),\qquad t \in [0,1],
\end{equation}
where $\omega_0 \in \mathbb{S}^{d-1}$ is the unit direction obtained above.

% [R8-MATH-034] Used the already established entire extension instead of inferring analyticity from derivative bounds alone.
By the choice of $y$, $|\phi(0)|\ge1$. Since $f$ is entire by Lemma~\ref{PWSDK}, the same formula defines an entire function $\phi$ on $\mathbb{C}$. The chain rule gives
\[
|\phi^{(m)}(0)| \leq \frac{C_0^{m} d^{\frac{3m}{2}} \max_{|\beta|=m} \big[\mu_k(S_j^1)\big]^{1/2} \|\partial^\beta f\|_{L^\infty(S_j^1)}}{\| f \|_{L_k^2(S_j^1)}} \quad \text{for all } m \geq 0.
\]

Combining the above derivative estimate with \eqref{ESB}, we obtain
\begin{equation}\label{E1006}
|\phi^{(m)}(0)| \leq C_2(d, R_+, G, k)(\frac{\mu_k(S_j^1)}{\inf_{x\in S_j^G}w(x)})^{\frac{1}{2}} (1 + N)^d \big( C_3(d, R_+, k) N \big)^m \quad \text{for all } m \geq 0.
\end{equation}
Since $\mu_k(S_j^1)\le |S_j^1|\sup_{x\in S_j^G}w(x)$ and $S_j^G\cap H_\delta=\emptyset$, we use the weight comparison
\[
\frac{\sup_{x\in S_j^G} w(x)}{\inf_{x\in S_j^G} w(x)} \le
\frac{\sup_{x\in S_j^1} w(x)}{\inf_{x\in S_j^1} w(x)} \le
% [R8-MATH-035] Used the diameter of the enlarged set, not that of a unit cube, and removed a spurious free alpha.
\prod_{\alpha\in R_+} \biggl(1+\frac{|\alpha|C_0\sqrt{d}}{\delta}\biggr)^{2k_\alpha}\le C(d,R_+,k)(1+N)^{2\gamma_k},
\]
and \eqref{E1006} therefore yields
\begin{equation}\label{E16}
|\phi^{(m)}(0)| \leq C_2(d, R_+, G, k) (1 + N)^{d_k} \big( C_3(d, R_+, k) N \big)^m \quad \text{for all } m \geq 0.
\end{equation}

% [R8-LANG-036] Removed a redundant analytic-extension argument; retained the quantitative use of the Taylor series.
The derivative bound controls the Taylor series of the entire function $\phi$ uniformly on every bounded disk. In particular, it will yield an exponential-in-$N$ bound on $\max_{|z|\le4}|\phi(z)|$.

We apply Lemma~\ref{L1} with $I = [0,1]$, $\hat{E} = I_{\omega_0}$ and $\Phi = \phi$. Together with \eqref{E11}, this yields a constant $C_4 = C_4(d, R_+,k,G)$ such that
\begin{equation}\label{E12}
\begin{aligned}
\sup_{t\in[0,1]} |\phi(t)| &\leq \left(\frac{C}{|I_{\omega_0}|}\right)^{\frac{\ln M}{\ln 2}} \sup_{t\in I_{\omega_0}} |\phi(t)| \\
&\leq \left( \frac{2 C C_0^d (d \pi)^{d/2}}{\gamma \, \Gamma\left(\frac{d}{2}\right)} \right)^{\frac{\ln M}{\ln 2}} \sup_{t\in I_{\omega_0}} |\phi(t)| \\
&\leq M^{C_4\left(1+\ln \frac{1}{\gamma}\right)} \sup_{t\in I_{\omega_0}} |\phi(t)|,
\end{aligned}
\end{equation}
where
\begin{equation}\label{E14}
M = \max_{|z|\leq 4} |\phi(z)|.
\end{equation}

We now make two observations. First, it follows directly from \eqref{E9} that
\[
|\phi(0)| = \frac{\big[\mu_k(S_j^1)\big]^{1/2} |f(y)|}{\| f \|_{L_k^2(S_j^1)}} = \frac{\big[\mu_k(S_j^1)\big]^{1/2}}{\| f \|_{L_k^2(S_j^1)}} \|f\|_{L^\infty(S_j^1)}.
\]

Second, by the definition of $I_{\omega_0}$, we have
\[
\sup_{t\in I_{\omega_0}} |\phi(t)| \leq \frac{\big[\mu_k(S_j^1)\big]^{1/2}}{\| f \|_{L_k^2(S_j^1)}} \|f\|_{L^\infty(E \cap S_j^1)}.
\]

Combining these two relations with \eqref{E12}, we deduce the crucial estimate
\begin{equation}\label{E13}
\|f\|_{L^\infty(S_j^1)} \leq M^{C_4\left(1+\ln \frac{1}{\gamma}\right)} \|f\|_{L^\infty(E \cap S_j^1)}.
\end{equation}

We now define a measurable subset
\[
E' := \left\{ x \in E \cap S_j^1 : |f(x)| \leq \frac{2}{\mu_k(E \cap S_j^1)} \int_{E \cap S_j^1} |f(x)| \, \mathrm{d}\mu_k(x) \right\}.
\]
% [R8-CHECK-P02B] Only the first Markov bound below is immediate. The second inequality needs a new density constant (including the factor 1/2), and |E'| must be bounded independently of delta and N. See the review's proposed lemma.
By Markov's inequality for the Dunkl measure and Remark~\ref{RR3}, we have
\[
\mu_k(E') \geq \frac{\mu_k(E \cap S_j^1)}{2} \ge \gamma\mu_k(S_j^1),
\]
Each unit cube contains a concentric ball of radius $1/2$. Using \eqref{EWBOUND}, we infer that there exists $\gamma_1>0$, depending only on $d$, $R_+$, $k$, and $\gamma$, such that
\[
|E'|\geq \gamma_1.
\]

Repeating the argument gives a constant $C_4=C_4(d,R_+,k,G)$ such that
\[
\|f\|_{L^\infty(S_j^1)} \leq M^{C_4\left(1+\ln \frac{1}{\gamma_1}\right)} \|f\|_{L^\infty(E' \cap S_j^1)}.
\]
Meanwhile, the definition of $E'$ gives
\[
\|f\|_{L^\infty(E' \cap S_j^1)} \leq \frac{2}{\mu_k(E \cap S_j^1)} \int_{E \cap S_j^1} |f(x)| \, \mathrm{d}\mu_k(x).
\]

% [R8-REF-037] Corrected the internal reference: the repeated argument on E-prime is needed, not the earlier bound on E.
Using the preceding estimate with $E'$, its defining threshold, and the Cauchy--Schwarz inequality, we obtain
\begin{align*}
\int_{S_j^1} |f|^2 \, \mathrm{d}\mu_k(x)
&\le \mu_k(S_j^1) \|f\|_{L^\infty(S_j^1)}^2 \\
&\le \mu_k(S_j^1) M^{2C'_4\left(\ln\frac{1}{\gamma_1}+1\right)} \|f\|_{L^\infty(E'\cap S_j^1)}^2 \\
&\le \mu_k(S_j^1) M^{2C'_4\left(\ln\frac{1}{\gamma_1}+1\right)}
\left( \frac{2}{\mu_k(E\cap S_j^1)} \right)^2
\left( \int_{E\cap S_j^1} |f(x)| \, \mathrm{d}\mu_k(x) \right)^2 \\
&\le 4 \mu_k(S_j^1) M^{2C'_4\left(\ln\frac{1}{\gamma_1}+1\right)}
\cdot \frac{1}{\mu_k(E\cap S_j^1)}
\int_{E\cap S_j^1} |f(x)|^2 \, \mathrm{d}\mu_k(x) \\
&\le \frac{4}{\gamma_1} M^{2C'_4\left(\ln\frac{1}{\gamma_1}+1\right)}
\int_{E\cap S_j^1} |f(x)|^2 \, \mathrm{d}\mu_k(x).
\end{align*}

We finally estimate the constant $M$ defined in \eqref{E14} via the derivative bound \eqref{E16}:
\begin{equation}\label{E17}
\begin{split}
M &= \max_{|z|\le 4} |\phi(z)| \\
&\le \max_{|z|\le 4} \sum_{m=0}^{\infty} \frac{|\phi^{(m)}(0)|}{m!} |z|^m \\
&\le \max_{|z|\le 4} \sum_{m=0}^{\infty} C_2(d, R_+, G, k) (1 + N)^{d_k} \left(d^{\frac{3}{2}} C_3(d, R_+, k) N\right)^m \frac{|z|^m}{m!} \\
&\leq e^{C_5(1+N)}
\end{split}
\end{equation}
for some positive constant $C_5 = C_5(d, R_+, k, G)$.

\emph{Step 5. Sum the local estimates.}

\begin{equation}\label{E20}
\begin{aligned}
\sum_{\substack{S_j^G \text{ is good}}} \|f\|^2_{L_k^2(S_j^G)}
&\leq |G| \sum_{\substack{S_j^G \text{ is good}}} \|f\|^2_{L_k^2(S_j^1)} \\
&\leq e^{C_6(1+N)} \sum_{\substack{S_j^G \text{ is good}}} \|f\|^2_{L_k^2(S_j^1 \cap E)} \\
&\leq e^{C_7(1+N)} \|f\|^2_{L_k^2(E)}.
\end{aligned}
\end{equation}
where $C_6 = C_6(d, R_+, G, k, \gamma)$ and $C_7 = C_6(d, R_+, G, k, \gamma) C_0$.

Decompose the global norm as
\begin{align*}
\|f\|^2_{L_k^2(\mathbb{R}^d)}
=\|f\|^2_{L_k^2(H_\delta)} + \|f\|^2_{L_k^2(\mathbb{R}^d\setminus H_\delta)}.
\end{align*}
By Remark~\ref{R1}, the norm on the layer $H_\delta$ satisfies
\[
\|f\|_{L_k^2(H_\delta)} \le \frac{1}{10} \|f\|_{L_k^2(\mathbb{R}^d)}.
\]
Substituting this bound gives
\[
\|f\|^2_{L_k^2(\mathbb{R}^d)}
\le \frac{100}{99} \|f\|^2_{L_k^2(\mathbb{R}^d\setminus H_\delta)}.
\]

The covering property of the family $\{S_j^G\}$ gives
\begin{align*}
\|f\|^2_{L_k^2(\mathbb{R}^d\setminus H_\delta)}
&\le C_0 \sum_{j} \|f\|^2_{L_k^2(S_j^G)} \\
&\le C_0 \sum_{\substack{S_j^G \text{ is good}}} \|f\|^2_{L_k^2(S_j^G)}
+ C_0 \sum_{\substack{S_j^G \text{ is bad}}} \|f\|^2_{L_k^2(S_j^G)}.
\end{align*}
Moreover, the contribution from \eqref{E19} is uniformly small:
\[
\sum_{\substack{S_j^G \text{ is bad}}} \|f\|^2_{L_k^2(S_j^G)}
\le \frac{C_0^2}{K^2} \|f\|^2_{L_k^2(\mathbb{R}^d)}.
\]

Combining the above inequalities with \eqref{E20}, we obtain
\[
\|f\|^2_{L_k^2(\mathbb{R}^d)}
\le \frac{100}{99} C_0 \left(
 e^{C_7(1+N)} \|f\|^2_{L_k^2(E)}
+ \frac{C_0^2}{K^2} \|f\|^2_{L_k^2(\mathbb{R}^d)}
\right).
\]

Rearranging terms,
\[
\left(1-\frac{100 C_0^3}{99 K^2}\right) \|f\|^2_{L_k^2(\mathbb{R}^d)}
\le \frac{100 C_0 e^{C_7(1+N)}}{99} \|f\|^2_{L_k^2(E)}.
\]

Choosing $K$ sufficiently large that $1-\frac{100 C_0^3}{99 K^2} > 0$, we finally conclude
\[
\|f\|_{L_k^2(\mathbb{R}^d)} \le e^{C_8(1+N)} \|f\|_{L_k^2(E)},
\]
where $C_8 = C_8(d, R_+, G, k, \gamma)$.

% [R8-MATH-038] Corrected the base-scale mismatch, the scaled-set notation, and the cancelled homogeneity factor; the base-case proof still has CHECK items.
We now pass from the base scale $1/d$ to a general side length $L>0$. Set $a=dL$, $E_a=a^{-1}E$, and $g(x)=f(ax)$. For every $x\in\mathbb{R}^d$, homogeneity gives
\[
\begin{aligned}
\mu_k\bigl(E_a\cap(x+d^{-1}Q)\bigr)
&=a^{-d_k}\mu_k\bigl(E\cap(ax+LQ)\bigr)\\
&\ge\gamma a^{-d_k}\mu_k(ax+LQ)
=\gamma\mu_k(x+d^{-1}Q).
\end{aligned}
\]
Thus $E_a$ is $(\gamma,1/d)$-thick with respect to $\mu_k$, precisely the scale used above. Moreover,
\[
\mathcal{F}_k g(\xi)=a^{-d_k}\mathcal{F}_k f(\xi/a),\qquad
\supp(\mathcal{F}_k g)\subset\overline{B(0,aN)},
\]
and
\[
\|g\|_{L_k^2(\mathbb{R}^d)}^2=a^{-d_k}\|f\|_{L_k^2(\mathbb{R}^d)}^2,\qquad
\|g\|_{L_k^2(E_a)}^2=a^{-d_k}\|f\|_{L_k^2(E)}^2.
\]
Applying the base-scale estimate to $g$ and cancelling $a^{-d_k/2}$ yields
\[
\|f\|_{L_k^2(\mathbb{R}^d)}\le e^{C(1+aN)}\|f\|_{L_k^2(E)}.
\]
Since $1+dLN\le\max\{1,dL\}(1+N)$, squaring this inequality and enlarging the constant gives \eqref{eq:LS}, with a constant depending only on $\gamma,L,d,R_+$, and $k$.
\end{proof}
 Combining Lemma~\ref{thm:suff} with Lemma~\ref{thm:necessity} yields the Logvinenko--Sereda theorem for the Dunkl transform, stated as Theorem~\ref{thm:main}.

\section{Completion of the Proof of Theorem~\ref{TOHS}}

We prove Theorem~\ref{TOHS} via the cyclic chain of implications:
\[
(\mathrm{i}) \Rightarrow (\mathrm{ii}) \Rightarrow (\mathrm{iii}) \Rightarrow (\mathrm{iv}) \Rightarrow (\mathrm{i}).
\]
We have already proved $(\mathrm{i}) \Rightarrow (\mathrm{ii})$ by Lemma~\ref{thm:suff} and $(\mathrm{ii}) \Rightarrow (\mathrm{i})$ by Lemma~\ref{thm:necessity}. We next prove $(\mathrm{ii}) \Rightarrow (\mathrm{iii})$ using the following lemma.
\begin{lemma}\label{L61}
Suppose that a measurable set $E \subset \mathbb{R}^d$ satisfies the spectral inequality \eqref{ESPEC}. Then $E$ satisfies the H\"older-type interpolation inequality \eqref{EHOLDER}, with
\[
C_{\mathrm{Hold}} = \frac{1}{1-\theta}\big(C_{\mathrm{spec}}+1\big)^2 + \ln 12.
\]
\end{lemma}
\begin{proof}
Let $E \subset \mathbb{R}^d$ satisfy the spectral inequality \eqref{ESPEC}. Fix $T > 0$, $\theta \in (0,1)$, and a solution $u$ of \eqref{EHEAT}. Write
\[
u_0(x) = u(0,x),\quad x \in \mathbb{R}^d.
\]
Then
\[
u(T,x) = \big(e^{T\Delta_k} u_0\big)(x) \quad \text{for all } x \in \mathbb{R}^d.
\]

Let
\[
u_1=\mathcal{F}_k^{-1} (\mathcal{F}_k u_0 \cdot \chi_{B(0,N)})
\]
\[
u_2=\mathcal{F}_k^{-1} (\mathcal{F}_k u_0 \cdot (1-\chi_{B(0,N)})),
\]
so $u_0=u_1+u_2$.

In this proof, write $\|v\|_2=\|v\|_{L_k^2(\mathbb{R}^d)}$,
$\|v\|_{2,E}=\|v\|_{L_k^2(E)}$ and $u(T)=u(T,\cdot)$. From this and \eqref{ESPEC}, we obtain
\begin{equation}\label{E51}
\begin{aligned}
\|u(T)\|_2^2
&\le 2\|e^{T\Delta_k}u_1\|_2^2+2\|e^{T\Delta_k}u_2\|_2^2\\
&\le 2e^{C_{\mathrm{spec}}(1+N)}\|e^{T\Delta_k}u_1\|_{2,E}^2
      +2\|e^{T\Delta_k}u_2\|_2^2\\
&\le 4e^{C_{\mathrm{spec}}(1+N)}\|u(T)\|_{2,E}^2
      +\bigl(2+4e^{C_{\mathrm{spec}}(1+N)}\bigr)\|e^{T\Delta_k}u_2\|_2^2.
\end{aligned}
\end{equation}
Since
\begin{align*}
\int_{\mathbb{R}^d} \big|(e^{T\Delta_k}u_2)(x)\big|^2 \mathrm{d}\mu_k(x)
&= \int_{\mathbb{R}^d} \big|e^{-T|\xi|^2} \chi_{\mathbb{R}^d\setminus B(0,N)}\mathcal{F}_k u_0(\xi)\big|^2 \mathrm{d}\mu_k(\xi) \\
&\le e^{-TN^2} \int_{\mathbb{R}^d} \big|\mathcal{F}_k u_0(\xi)\big|^2 \mathrm{d}\mu_k(\xi) \\
&= e^{-TN^2} \int_{\mathbb{R}^d} |u_0(x)|^2 \mathrm{d}\mu_k(x),
\end{align*}
it follows from \eqref{E51} that
\begin{equation}\label{E52}
\begin{aligned}
\|u(T)\|_2^2
&\le 4e^{C_{\mathrm{spec}}(1+N)}\|u(T)\|_{2,E}^2
  +\bigl(2+4e^{C_{\mathrm{spec}}(1+N)}\bigr)e^{-TN^2}\|u_0\|_2^2\\
&\le 6e^{C_{\mathrm{spec}}}\left(
 e^{C_{\mathrm{spec}}N}\|u(T)\|_{2,E}^2
 +e^{C_{\mathrm{spec}}N-TN^2}\|u_0\|_2^2\right).
\end{aligned}
\end{equation}

Given $\varepsilon \in (0,1)$, choose $N = N(\varepsilon)$ so that
\[
\exp\big[C_{\mathrm{spec}}N - TN^2\big] = \varepsilon.
\]
(This is possible because the set $\{C_{\mathrm{spec}}s - Ts^2: s > 0\}$ contains $(-\infty,0]$.) With this choice of $N$,
\[
N = \frac{C_{\mathrm{spec}} + \sqrt{C_{\mathrm{spec}}^2 + 4T\ln\frac{1}{\varepsilon}}}{2T}
\le \frac{1}{T}\left(C_{\mathrm{spec}} + \sqrt{T\ln\frac{1}{\varepsilon}}\right).
\]
For the previously fixed $\theta \in (0,1)$, it follows that
\begin{align*}
\exp\big[C_{\mathrm{spec}}N\big]
&\le \exp\left[\frac{C_{\mathrm{spec}}^2}{T}\right]
\exp\left[\frac{C_{\mathrm{spec}}}{\sqrt{T}}\sqrt{\ln\frac{1}{\varepsilon}}\right] \\
&\le \exp\left[\frac{C_{\mathrm{spec}}^2}{T}\right]
\exp\left[\frac{1-\theta}{\theta}\ln\frac{1}{\varepsilon} + \frac{\theta}{1-\theta}\frac{C_{\mathrm{spec}}^2}{T}\right]
= \exp\left[\frac{C_{\mathrm{spec}}^2}{(1-\theta)T}\right] \varepsilon^{-\frac{1-\theta}{\theta}}.
\end{align*}

From this and \eqref{E52}, we find that for every $\varepsilon \in (0,1)$,
\[
\int_{\mathbb{R}^d} |u(T,x)|^2 \mathrm{d}\mu_k(x)
\le 6e^{C_{\mathrm{spec}}}
\left(
e^{\frac{C_{\mathrm{spec}}^2}{(1-\theta)T}} \varepsilon^{-\frac{1-\theta}{\theta}}
\int_{E} |u(T,x)|^2 \mathrm{d}\mu_k(x)
+ \varepsilon \int_{\mathbb{R}^d} |u_0(x)|^2 \mathrm{d}\mu_k(x)
\right).
\]
If $u_0=0$, the conclusion is immediate. If $\|u(T)\|_{2,E}=0$,
let $\varepsilon\downarrow0$ in the preceding estimate. If
$\|u(T)\|_{2,E}^2/\|u_0\|_2^2=1$, use the contraction property.
In the remaining case the ratio lies in $(0,1)$, so choosing
\[
\varepsilon = \left(
\frac{\int_{E} |u(T,x)|^2 \mathrm{d}\mu_k(x)}{\int_{\mathbb{R}^d} |u_0(x)|^2 \mathrm{d}\mu_k(x)}
\right)^{\theta},
\]
we obtain
\begin{align*}
\|u(T)\|_2^2
&\le 12e^{C_{\mathrm{spec}}}
 e^{\frac{C_{\mathrm{spec}}^2}{(1-\theta)T}}
 \|u(T)\|_{2,E}^{2\theta}\|u_0\|_2^{2(1-\theta)}\\
&\le e^{\left(\frac{(C_{\mathrm{spec}}+1)^2}{1-\theta}+\ln12\right)
                   \left(1+\frac1T\right)}
 \|u(T)\|_{2,E}^{2\theta}\|u_0\|_2^{2(1-\theta)},
\end{align*}
which leads to \eqref{EHOLDER} with
\[
C_{\mathrm{Hold}} = \frac{1}{1-\theta}(C_{\mathrm{spec}}+1)^2 + \ln 12.
\]
\end{proof}

Next, we prove $(\mathrm{iii}) \Rightarrow (\mathrm{iv})$  by Lemma~\ref{lem:obs}.

\begin{lemma}\label{lem:obs}
Suppose that a measurable set $E \subset \mathbb{R}^d$ admits a positive constant $C_{\mathrm{Hold}} = C_{\mathrm{Hold}}(d,R,k,E)$ such that for any $T>0$,
\begin{equation}\label{E53}
\int_{\mathbb{R}^d} |u(T,x)|^2 \, \mathrm{d}\mu_k(x)
\le e^{C_{\mathrm{Hold}}\left(1+\frac{1}{T}\right)}
\left(\int_{E} |u(T,x)|^2 \, \mathrm{d}\mu_k(x)\right)^{1/2}
\left(\int_{\mathbb{R}^d} |u(0,x)|^2 \, \mathrm{d}\mu_k(x)\right)^{1/2},
\end{equation}
whenever $u$ solves the heat equation \eqref{EHEAT}.

Then, for each $T>0$ and each measurable subset $F \subset (0,T)$ of positive measure, there exists a positive constant $C_{\mathrm{obs}} = C_{\mathrm{obs}}(d,T,F,C_{\mathrm{Hold}})$ such that whenever $u$ solves \eqref{EHEAT},
\begin{equation}\label{E54}
\int_{\mathbb{R}^d} |u(T,x)|^2 \, \mathrm{d}\mu_k(x)
\le C_{\mathrm{obs}} \int_{F} \int_{E} |u(s,x)|^2 \, \mathrm{d}\mu_k(x)\, \mathrm{d}s.
\end{equation}
In particular, if $F=(0,T)$, then the constant $C_{\mathrm{obs}}$ in \eqref{E54} can be chosen as
\[
C_{\mathrm{obs}} = \exp\bigl[36(1+3C_{\mathrm{Hold}})(1+1/T)\bigr].
\]
\end{lemma}

\begin{proof}
Assume that $E \subset \mathbb{R}^d$ satisfies \eqref{E53}. Fix $T > 0$ and a measurable subset $F\subset (0,T)$ of positive measure. % [R8-LANG-P2-19] Polished the exposition without changing the argument.
Young's inequality applied to \eqref{E53} gives, for every $t>0$ and $\varepsilon>0$,
\begin{equation}\label{E55}
\int_{\mathbb{R}^d} |u(t,x)|^2 \, \mathrm{d}\mu_k(x)
\le \frac{1}{\varepsilon}e^{2C_{\mathrm{Hold}}\left(1+\frac{1}{t}\right)}
\int_{E} |u(t,x)|^2 \, \mathrm{d}\mu_k(x)
+ \varepsilon \int_{\mathbb{R}^d} |u_0(x)|^2 \, \mathrm{d}\mu_k(x),
\end{equation}
where $u_0(x)=u(0,x)$.

By a time translation, the preceding inequality \eqref{E55} implies that for all $0 < t_1 < t_2$ and $\varepsilon > 0$,
\begin{equation}\label{E56}
\int_{\mathbb{R}^d} |u(t_2,x)|^2 \, \mathrm{d}\mu_k(x)
\le \frac{1}{\varepsilon}e^{2C_{\mathrm{Hold}}\left(1+\frac{1}{t_2-t_1}\right)}
\int_{E} |u(t_2,x)|^2 \, \mathrm{d}\mu_k(x)
+ \varepsilon \int_{\mathbb{R}^d} |u(t_1,x)|^2 \, \mathrm{d}\mu_k(x).
\end{equation}

Let $\ell$ be a Lebesgue density point of $F$. According to \cite[Proposition~2.1]{KDP}, for each $\lambda \in (0,1)$ there exists a sequence $\{\ell_m\}_{m=1}^\infty \subset (\ell,T)$ such that for every $m \in \mathbb{N}^+$,
\begin{equation}\label{E57}
\ell_{m+1} - \ell = \lambda^m(\ell_1 - \ell)
\end{equation}
and
\begin{equation}\label{E58}
\bigl|F \cap (\ell_{m+1},\ell_m)\bigr| \ge \frac{1}{3}(\ell_m - \ell_{m+1}).
\end{equation}

Fix $m \in \mathbb{N}^+$ and choose $s$ such that
\[
0 < \ell_{m+2} < \ell_{m+1} \le s < \ell_m < T.
\]
Using \eqref{E56} with $t_1 = \ell_{m+2}$ and $t_2 = s$, and noting that
\[
\int_{\mathbb{R}^d} |u(\ell_m,x)|^2 \, \mathrm{d}\mu_k(x) \le \int_{\mathbb{R}^d} |u(s,x)|^2 \, \mathrm{d}\mu_k(x)
\]
and $\ell_{m+1}-\ell_{m+2}\le s-\ell_{m+2}$, we obtain
\begin{equation}\label{E59}
\int_{\mathbb{R}^d} |u(\ell_m,x)|^2 \, \mathrm{d}\mu_k(x)
\le \frac{1}{\varepsilon}e^{2C_{\mathrm{Hold}}\left(1+\frac{1}{\ell_{m+1}-\ell_{m+2}}\right)}
\int_{E} |u(s,x)|^2 \, \mathrm{d}\mu_k(x)
+ \varepsilon \int_{\mathbb{R}^d} |u(\ell_{m+2},x)|^2 \, \mathrm{d}\mu_k(x).
\end{equation}

Integrating \eqref{E59} with respect to $s$ over $F\cap(\ell_{m+1},\ell_m)$ gives
\begin{equation}\label{E510}
\begin{aligned}
\int_{\mathbb{R}^d}|u(\ell_m,x)|^2\,\mathrm{d}\mu_k(x)
&\le \varepsilon\int_{\mathbb{R}^d}|u(\ell_{m+2},x)|^2\,\mathrm{d}\mu_k(x)\\
&\quad+\frac{e^{2C_{\mathrm{Hold}}\left(1+\frac1{\ell_{m+1}-\ell_{m+2}}\right)}}
 {\varepsilon\,|F\cap(\ell_{m+1},\ell_m)|}\\
&\qquad\times\int_{F\cap(\ell_{m+1},\ell_m)}\int_E
 |u(s,x)|^2\,\mathrm{d}\mu_k(x)\,\mathrm{d}s.
\end{aligned}
\end{equation}

By \eqref{E58},
\[
|F\cap(\ell_{m+1},\ell_m)| \ge \frac13(\ell_m-\ell_{m+1}) \ge \frac13 e^{-\frac{1}{\ell_m-\ell_{m+1}}},
\]
we deduce from \eqref{E510} that
\begin{equation}\label{E511}
\begin{aligned}
\int_{\mathbb{R}^d} |u(\ell_m,x)|^2 \, \mathrm{d}\mu_k(x)
&\le \varepsilon \int_{\mathbb{R}^d} |u(\ell_{m+2},x)|^2 \, \mathrm{d}\mu_k(x) \\
&\quad + \frac{3}{\varepsilon}e^{\frac{1}{\ell_m-\ell_{m+1}}+2C_{\mathrm{Hold}}\left(1+\frac{1}{\ell_{m+1}-\ell_{m+2}}\right)}
\int_{F\cap(\ell_{m+1},\ell_m)}\int_{E} |u(s,x)|^2 \, \mathrm{d}\mu_k(x)\, \mathrm{d}s.
\end{aligned}
\end{equation}

Furthermore, from \eqref{E57} we have
\begin{equation}\label{E512}
\ell_m - \ell_{m+1} = \frac{1}{1+\lambda}(\ell_m - \ell_{m+2})
\end{equation}
and
\begin{equation}\label{E513}
\ell_{m+1} - \ell_{m+2} = \frac{\lambda}{1+\lambda}(\ell_m - \ell_{m+2}).
\end{equation}
Substituting \eqref{E512} and \eqref{E513} into \eqref{E511} yields
\begin{equation} \label{E514}
\begin{aligned}
\int_{\mathbb{R}^d} |u(\ell_m,x)|^2 \, \mathrm{d}\mu_k(x)
&\le \varepsilon \int_{\mathbb{R}^d} |u(\ell_{m+2},x)|^2 \, \mathrm{d}\mu_k(x) \\
&\quad + 3e^{2C_{\mathrm{Hold}}} \frac{1}{\varepsilon} e^{\frac{C'}{\ell_m-\ell_{m+2}}}
\int_{F\cap(\ell_{m+1},\ell_m)} \int_{E} |u(s,x)|^2 \, \mathrm{d}\mu_k(x)\, \mathrm{d}s,
\end{aligned}
\end{equation}
where
\begin{equation}\label{E515}
C' := 1+\lambda + \frac{2C_{\mathrm{Hold}}(1+\lambda)}{\lambda}.
\end{equation}

We rewrite \eqref{E514} as
\begin{equation}\label{E516}
\begin{aligned}
\varepsilon e^{-\frac{C'}{\ell_m-\ell_{m+2}}} \int_{\mathbb{R}^d} |u(\ell_m,x)|^2 \, \mathrm{d}\mu_k(x)
&- \varepsilon^2 e^{-\frac{C'}{\ell_m-\ell_{m+2}}} \int_{\mathbb{R}^d} |u(\ell_{m+2},x)|^2 \, \mathrm{d}\mu_k(x) \\
&\le 3e^{2C_{\mathrm{Hold}}} \int_{F\cap(\ell_{m+1},\ell_m)} \int_{E} |u(s,x)|^2 \, \mathrm{d}\mu_k(x)\, \mathrm{d}s.
\end{aligned}
\end{equation}

Now fix $\lambda \in (1/\sqrt{2},1)$ and set $\mu := \frac{1}{2-\lambda^{-2}} > 1$. In \eqref{E516}, choose
\[
\varepsilon = \exp\left[-\frac{(\mu-1)C'}{\ell_m-\ell_{m+2}}\right].
\]
Then we obtain
\begin{equation}\label{E517}
\begin{aligned}
e^{-\frac{\mu C'}{\ell_m-\ell_{m+2}}}\int_{\mathbb{R}^d} |u(\ell_m,x)|^2\, \mathrm{d}\mu_k(x)
&- e^{-\frac{(2\mu-1)C'}{\ell_m-\ell_{m+2}}}\int_{\mathbb{R}^d} |u(\ell_{m+2},x)|^2\, \mathrm{d}\mu_k(x) \\
&\le 3e^{2C_{\mathrm{Hold}}}
\int_{F\cap(\ell_{m+1},\ell_m)}\int_{E} |u(s,x)|^2\, \mathrm{d}\mu_k(x)\, \mathrm{d}s.
\end{aligned}
\end{equation}

A direct calculation shows
\begin{equation}\label{E518}
\exp\left[-\frac{(2\mu-1)C'}{\ell_m-\ell_{m+2}}\right]
= \exp\left[-\frac{\mu C'}{\lambda^2(\ell_m-\ell_{m+2})}\right].
\end{equation}
Since
\[
\ell_{m+2}-\ell_{m+4} = \lambda^2(\ell_m-\ell_{m+2}),
\]
we infer from \eqref{E517} and \eqref{E518} that
\[
\begin{aligned}
&e^{-\frac{\mu C'}{\ell_m-\ell_{m+2}}}\int_{\mathbb{R}^d} |u(\ell_m,x)|^2\, \mathrm{d}\mu_k(x)
-e^{-\frac{\mu C'}{\ell_{m+2}-\ell_{m+4}}}\int_{\mathbb{R}^d} |u(\ell_{m+2},x)|^2\, \mathrm{d}\mu_k(x) \\
&\le 3e^{2C_{\mathrm{Hold}}}
\int_{F\cap(\ell_{m+1},\ell_m)}\int_{E} |u(s,x)|^2\, \mathrm{d}\mu_k(x)\, \mathrm{d}s.
\end{aligned}
\]
Summing the above inequality over all odd $m$ yields
\[
\begin{aligned}
e^{-\frac{\mu C'}{\ell_1-\ell_3}}\int_{\mathbb{R}^d} |u(\ell_1,x)|^2\, \mathrm{d}\mu_k(x)
&\le 3e^{2C_{\mathrm{Hold}}}\sum_{\substack{m\ge1\\ m\text{ odd}}}\int_{F\cap(\ell_{m+1},\ell_m)}\int_{E} |u(s,x)|^2\, \mathrm{d}\mu_k(x)\, \mathrm{d}s \\
&\le 3e^{2C_{\mathrm{Hold}}}\int_{F\cap(\ell,\ell_1)}\int_{E} |u(s,x)|^2\, \mathrm{d}\mu_k(x)\, \mathrm{d}s \\
&\le 3e^{2C_{\mathrm{Hold}}}\int_{F}\int_{E} |u(s,x)|^2\, \mathrm{d}\mu_k(x)\, \mathrm{d}s.
\end{aligned}
\]

Thus,
\begin{equation}\label{E519}
\int_{\mathbb{R}^d} |u(T,x)|^2\, \mathrm{d}\mu_k(x)
\le \int_{\mathbb{R}^d} |u(\ell_1,x)|^2\, \mathrm{d}\mu_k(x)
\le 3e^{2C_{\mathrm{Hold}}}e^{\frac{\mu C'}{\ell_1-\ell_3}}
\int_{F}\int_{E} |u(s,x)|^2\, \mathrm{d}\mu_k(x)\, \mathrm{d}s,
\end{equation}
which gives \eqref{E54} with
\[
C_{\mathrm{obs}}=3\exp\left[2C_{\mathrm{Hold}}+\frac{\mu C'}{\ell_1-\ell_3}\right].
\]

Finally, in the case $F=(0,T)$, we choose
\[
\ell_1=\frac{2T}{3},\quad \ell=\frac{T}{3},\quad \lambda=\sqrt{\frac{2}{3}}.
\]
Then (see \eqref{E515})
\[
\ell_1-\ell_3=\frac{T}{9},\qquad \mu=2,\qquad C'\le 2+6C_{\mathrm{Hold}}.
\]
From \eqref{E519} we derive
\[
\begin{aligned}
\int_{\mathbb{R}^d}|u(T,x)|^2\, \mathrm{d}\mu_k(x)
&\le 3e^{2C_{\mathrm{Hold}}}e^{\frac{36(1+3C_{\mathrm{Hold}})}{T}}
\int_{0}^{T}\int_{E}|u(s,x)|^2\, \mathrm{d}\mu_k(x)\, \mathrm{d}s \\
&\le e^{36(1+3C_{\mathrm{Hold}})\left(1+\frac{1}{T}\right)}
\int_{0}^{T}\int_{E}|u(s,x)|^2\, \mathrm{d}\mu_k(x)\, \mathrm{d}s.
\end{aligned}
\]

This completes the proof.
\end{proof}

Finally, we prove $(\mathrm{iv}) \Rightarrow (\mathrm{i})$.

\begin{lemma}\label{LFIN}
Assume $k_\alpha>0$ for every root. Suppose that a measurable set $E \subset \mathbb{R}^d$ satisfies the observability inequality \eqref{EOBS}. Then $E$ is $\gamma$-thick at scale $L$ with respect to $\mu_k$ for some $\gamma>0$ and $L>0$.
\end{lemma}

\begin{proof}
% [R8-LANG-P2-20] Polished the exposition without changing the argument.
It suffices to use observability at $T=1$. We enlarge the constant, if necessary, so that $C_{\mathrm{obs}}\ge1$. Write $V(y,r)=\mu_k(B(y,r))$. The estimate in \cite[Remark~3.3, equation~(3.12)]{JDA1}, together with symmetry, gives
\[
h_t(x,y)\le\frac{C}{V(y,\sqrt t)}
\left(1+\frac{|x-y|}{\sqrt t}\right)^{-2}e^{-c\,d_G(x,y)^2/t}.
\]
By \eqref{EWBOUND}, $w(y)/V(y,\sqrt t)\le Ct^{-d/2}$. Also
\[
\int_{\mathbb{R}^d}e^{-c\,d_G(x,y)^2/t}\,\mathrm{d}y
\le\sum_{\sigma\in G}\int_{\mathbb{R}^d}e^{-c|x-\sigma y|^2/t}\,\mathrm{d}y
\le C t^{d/2}.
\]
Consequently,
\begin{equation}\label{eq:heat-tail-review}
\sup_{x\in\mathbb{R}^d}\sup_{0<t\le1}
\int_{|x-y|\ge a}h_t(x,y)\,\mathrm{d}\mu_k(y)
\le C(1+a)^{-2}=:\varepsilon_a.
\end{equation}
Choose $M>0$ so that $q:=\varepsilon_{M/2}\le\min\{1/8,1/(8C_{\mathrm{obs}})\}$. For an arbitrary center $x_0$, set
\[
Q_i=x_0+(L+iM)Q\quad(i=0,1,2),\qquad u_0=\chi_{Q_1},\qquad u(t)=H_tu_0.
\]
Positivity and conservation of mass imply $0\le u\le1$. By \eqref{eq:heat-tail-review},
\begin{equation}\label{E531}
u(1,x)\ge1-q\quad(x\in Q_0).
\end{equation}
Symmetry, Tonelli's theorem, and $u^2\le u$ give, uniformly for $0<t\le1$,
\begin{equation}\label{E532}
\int_{Q_2^c}u(t,x)^2\,\mathrm{d}\mu_k(x)
\le\int_{Q_1}\int_{Q_2^c}h_t(x,y)\,\mathrm{d}\mu_k(x)\mathrm{d}\mu_k(y)
\le q\mu_k(Q_1).
\end{equation}
Observability now yields
\begin{equation}\label{E533}
(1-q)^2\mu_k(Q_0)\le C_{\mathrm{obs}}\mu_k(E\cap Q_2)
+C_{\mathrm{obs}}q\mu_k(Q_1).
\end{equation}
Choose $L$ sufficiently large in Lemma~\ref{L4001}, with an increment of $2M$, so that $\mu_k(Q_0)\ge(3/4)\mu_k(Q_2)$ uniformly in $x_0$. Since $\mu_k(Q_1)\le\mu_k(Q_2)$,
\[
C_{\mathrm{obs}}\mu_k(E\cap Q_2)
\ge\left[\frac34\left(\frac78\right)^2-\frac18\right]\mu_k(Q_2)
=\frac{115}{256}\mu_k(Q_2)\ge\frac14\mu_k(Q_2).
\]
Thus $E$ is thick at side length $L+2M$, with thickness parameter $1/(4C_{\mathrm{obs}})$.
\end{proof}

\begin{proof}[Proof of Theorem~\ref{TOHS}]
We use the following cycle of implications:
\[
(\mathrm{i}) \Rightarrow (\mathrm{ii}) \Rightarrow (\mathrm{iii}) \Rightarrow (\mathrm{iv}) \Rightarrow (\mathrm{i}).
\]
These implications correspond, respectively, to Lemmas~\ref{thm:suff}, \ref{L61}, \ref{lem:obs}, and \ref{LFIN}.
\end{proof}

% [R8-REF-SORT] Sorted by first-author surname; de Jeu under D. Same authors: year, then title. Citation keys retained.

\end{document}